\documentclass[12pt]{article}

\usepackage{mathptmx}
\usepackage[12pt]{moresize}

\usepackage{genyoungtabtikz}
\usepackage{ytableau}

\usepackage{amssymb}%
\usepackage{bbm}%
\usepackage{mathtools}
\usepackage{amsmath}%

\usepackage{wasysym}

\usepackage{amsthm}%

\usepackage{tikz-cd}

\usepackage{xcolor}%
\usepackage[utf8]{inputenc}

\usepackage{pst-node}
\DeclareMathOperator\id{id}

\allowdisplaybreaks%
\usepackage{stmaryrd}

\makeatletter
\newcommand{\xMapsto}[2][]{\ext@arrow 0599{\Mapstofill@}{#1}{#2}}
\def\Mapstofill@{\arrowfill@{\Mapstochar\Relbar}\Relbar\Rightarrow}
\makeatother

{\theoremstyle{plain}%
 \newtheorem{theorem}{Theorem}

 \newtheorem{lemma}{Lemma}%
}
{\theoremstyle{remark}

\newtheorem{remark}{Remark}
}
{\theoremstyle{definition}
\newtheorem{definition}{Definition}
\newtheorem{example}{Example}
}

\usepackage{xcolor,colortbl}
\definecolor{green}{rgb}{0.1,0.1,0.1}

\begin{document}

\begin{center}
{\Large A Combinatorial Hopf Algebra on partition diagrams}
\end{center}

\begin{center}
{\textsc{John M. Campbell} }
\end{center}

\begin{abstract}
 We introduce a Combinatorial Hopf Algebra (CHA) with bases indexed by the partition diagrams indexing the bases for partition algebras. By analogy with 
 the operation $H_{\alpha} H_{\beta} = H_{\alpha \cdot \beta}$ for the complete homogeneous basis of the CHA $ \textsf{NSym}$ given by concatenating 
 compositions $\alpha$ and $\beta$, we mimic this multiplication rule by setting $\textsf{H}_{\pi} \textsf{H}_{\rho} = \textsf{H}_{\pi \otimes \rho}$ for 
 partition diagrams $\pi$ and $\rho$ and for the horizontal concatenation $\pi \otimes \rho$ of $ \pi$ and $\rho$. This gives rise to a free, graded algebra $ 
 \textsf{ParSym}$, which we endow with a CHA structure by lifting the CHA structure of $ \textsf{NSym}$ using an analogue, for partition diagrams, of 
 near-concatenations of integer compositions. Unlike the Hopf algebra $\textsf{NCSym}$ on set partitions, the new CHA $\textsf{ParSym}$ projects onto $ 
 \textsf{NSym}$ in natural way via a ``forgetful'' morphism analogous to the projection of $\textsf{NSym}$ onto its commutative counterpart 
 $\textsf{Sym}$. We prove, using the Boolean transform for the sequence $(B_{2n} : n 
 \in \mathbb{N})$ of even-indexed Bell numbers, an analogue of Comtet's generating function for the sequence counting irreducible permutations, yielding a 
 formula for the number of generators in each degree for $\textsf{ParSym}$, and we prove, using a sign-reversing involution, an evaluation for the 
 antipode for $\textsf{ParSym}$. An advantage of our CHA being defined on partition diagrams in full generality, 
 in contrast to a previously defined Hopf algebra on uniform block permutations, 
 is given by how the coproduct operation we have defined for $\textsf{ParSym}$ 
 is such that the usual diagram subalgebras of partition algebras naturally give rise to 
 \emph{Hopf} subalgebras of $\textsf{ParSym}$ by restricting the indexing sets of the graded components 
 to diagrams of a specified form, as in with perfect matchings, partial permutations, planar diagrams, etc. 
\end{abstract}



\section{Introduction}
 The Hopf algebra $\textsf{NSym}$ of noncommutative symmetric functions introduced in \cite{GelfandKrobLascouxLeclercRetakhThibon1995} 
 continues to be applied 
 in important ways within combinatorics and many other areas. The underlying algebra of the bialgebra $\textsf{NSym}$ is such that 
\begin{equation}\label{displayNSym}
 \textsf{NSym} = \mathbbm{k}\langle H_{1}, H_{2}, \ldots \rangle, 
\end{equation}
 providing a noncommutative companion to 
\begin{equation}\label{displaySym}
 \textsf{Sym} = \mathbbm{k}[h_{1}, h_{2}, \ldots]. 
\end{equation}
 The free algebra structure indicated in \eqref{displayNSym}
 naturally gives rise to the complete homogeneous basis 
\begin{equation}\label{completefamily}
 \{ H_{\alpha} : \alpha \in \mathcal{C} \}
\end{equation}
 of $\textsf{NSym}$, subject to the multiplication rule 
\begin{equation}\label{homogeneousmult}
 H_{\alpha} H_{\beta} = H_{\alpha \cdot \beta}, 
\end{equation}
 letting $\mathcal{C}$ denote the set of all integer compositions, and letting $\alpha \cdot \beta$ $= $ $ (\alpha_{1}$, $ \alpha_{2}$, $ \ldots$, $ 
 \alpha_{\ell(\alpha)}$, $ \beta_{1}$, $ \beta_{2}$, $ \ldots$, $ \beta_{\ell(\beta)})$ denote the concatenation of $\alpha$ $= $ $ (\alpha_{1}$, 
 $ \alpha_{2}$, $ \ldots$, $ \alpha_{\ell(\alpha)})$ and $\beta$ $ = $ $ (\beta_{1}$, $ \beta_{2}$, $ \ldots$, $ \beta_{\ell(\beta)})$. Given a class of 
 combinatorial objects, if we construct an algebra with bases indexed by such objects according to a multiplication rule given by a concatenation or 
 concatenation-type operation, and without further conditions being imposed on the multiplication, this gives rise to a graded algebra structure closely 
 related to \eqref{displayNSym}, and motivates the construction of new Hopf algebras satisfying the axioms for Combinatorial Hopf Algebras (CHAs), as 
 introduced by Aguiar, Bergeron, and Sottile \cite{AguiarBergeronSottile2006}. We introduce a CHA, which we denote as $\textsf{ParSym}$, with 
 bases indexed by the combinatorial objects indexing the bases of partition algebras and with a multiplication rule defined by analogy with 
 \eqref{homogeneousmult}. In this regard, the horizontal concatenation operation $\otimes$ on partition diagrams is a natural operation to use to form a 
 weight function for a graded algebra with bases indexed by partition diagrams. Our use of the $\otimes$ operation 
 on the basis elements of partition algebras is also inspired by the uses of this operation in 
 the context of the character theory of partition algebras \cite{Halverson2001}, 
 and by how product operations for CHAs on graphs and on many other combinational 
 objects are often defined via concatenation or closely related operations, 
 as in with the disjoint union of graphs. 

 Hopf algebras and partition algebras play important roles in many different areas within algebraic combinatorics and physics. However, it appears that 
 partition algebras have not previously been endowed with any Hopf algebra or bialgebra structures. The foregoing considerations motivate the problem of 
 introducing a Hopf algebra structure on the combinatorial objects indexing the bases of partition algebras. Aguiar and Orellana \cite{AguiarOrellana2008} 
 have introduced a Hopf algebra on uniform block permutations, which are special cases of partition diagrams, and this motivates the construction of a Hopf 
 algebra on partition diagrams in full generality. The Aguiar--Orellana Hopf algebra 
 being free also motivates our constructing, as below, a free Hopf algebra on partition diagrams in full generality. 
 It seems that past references related to the Aguiar--Orellana algebra \cite{AguiarOrellana2008}, including 
 \cite{CheballahGiraudoMaurice2015,DaughertyHerbrich2014,Fisher2010,Foissy2007,Foissy2012,Maltcev2007,Maurice2013}, have not concerned 
 Hopf algebras on partition diagrams in full generality. 

 There is a great amount of literature on Hopf algebras on families of graphs; see 
 \cite{AguiarOrellana2008,ArcisMarquez2022,BergeronGonzalezdLeonLiPangVargas2023,DuchampLuqueNovelliTolluToumazet2011,Foissy2002,FoissyUnterberger2013,GrossmanLarson1989,Kreimer2010,Kreimer1998,LodayRonco1998,Schmitt1993,WangXuGao2020,ZhangXuGuo2022} 
 and many related references on CHAs on graphs and graph-like objects. This past literature motivates the problem of constructing a CHA on partition 
 diagrams, which form a naturally occurring family of simple graphs that are often used within statistical mechanics and within representation theory. 
 The coproduct operation for partition diagrams that we introduce depends on 
 the labeling system for partition diagrams together with a binary operation 
 $\bullet$ that provides an analogue of near-concatenation 
 for partition diagrams, as opposed to integer compositions. 
 This is in contrast to comultiplication operations for 
 previously studied Hopf algebras on graphs, 
 such as the Hopf algebra 
 $\mathcal{G}$ on finite graphs introduced in 
 \cite{Schmitt1994} and later studied in references such as \cite{HumpertMartin2012}. 
 The Hopf algebra of diagrams due to Duchamp et al.\ \cite{DuchampLuqueNovelliTolluToumazet2011} is not related to diagram algebras or partition algebras, 
 but is based on a family of combinatorial objects that are defined and denoted in something of a similar way relative to partition diagrams, which further 
 motivates the interest in constructing a Hopf algebra with bases indexed by partition diagrams. See also the work on Hopf algebras on dissection 
 diagrams due to Dupont \cite{Dupont2014} and Mammez \cite{Mammez2020}. 

 In Section \ref{subsectionNCSym} below, we consider how our new CHA $\textsf{ParSym}$ relates to the Hopf algebra $\textsf{NCSym}$ 
 \cite{BergeronReutenauerRosasZabrocki2008}, the bases of which are indexed by set partitions in a similar way, relative to the bases of 
 $\textsf{ParSym}$, and Section \ref{subsectionNCSym} also highlights some of the main points of interest concerning the new Hopf algebra 
 $\textsf{ParSym}$, in relation to $\textsf{NCSym}$ and otherwise. Beforehand, we are to briefly review preliminaries on partition diagrams, 
 as in Section \ref{subsectiondiagrams} below. The main, nonintroductory sections of our article are summarized as below. 

 \ 

 - In Section \ref{sectionnewHopf}, we define $\textsf{ParSym}$ as a graded algebra 
 and determine an irreducible generating set for $\textsf{ParSym}$, 
 and we determine the number of generators in each degree for $\textsf{ParSym}$. 

 \ 

 - In Section \ref{sectionCHA}, we introduce a CHA structure on $\textsf{ParSym}$, 
 and we prove a graph-theoretic property concerning an analogue of near-concatenation, 
 to construct a CHA projection morphism from $\textsf{ParSym}$ to $\textsf{NSym}$, 
 analogous to the projection of $\textsf{NSym}$ onto $\textsf{Sym}$. 
 We introduce an analogue of the $E$-generators of $\textsf{NSym}$
 to define an antipode antihomomorphism on $\textsf{ParSym}$, 
 and we prove that the required antipode axioms are satisfied, using a sign-reversing involution. 
 Our CHA morphism from $\textsf{ParSym}$ to $\textsf{NSym}$ allows us to evaluate the unique 
 CHA morphism from $\textsf{ParSym}$ to $\textsf{QSym}$. 

 \ 

 - In Section \ref{sectionDiagramHopf}, we show how each of the families of partition diagrams 
 associated with what may be regarded as the main or most important subalgebras of $\mathbb{C}A_k(n)$ 
 naturally gives rise to a Hopf subalgebra of $\textsf{ParSym}$. 

 \ 

 - In Section \ref{sectionConclusion}, we conclude with a number of further research areas to explore 
 related to the CHA $\textsf{ParSym}$ introduced in this article. 
 
\subsection{Partition diagrams}\label{subsectiondiagrams}
 For a set partition $\pi$ of $\{ 1, 2, \ldots, k, 1', 2', \ldots, k' \}$, we denote $\pi$ as a graph $G$ by arranging the elements of $\{ 1, 2, \ldots, k \}$ into 
 a top row and arranging the members of $\{ 1', 2', \ldots, k' \}$ into a bottom row, with members in $\pi$ forming the connected components of $G$. We 
 consider any two graphs $G$ and $G'$ on $\{ 1$, $2$, $ \ldots$, $k$, $1'$, $ 2'$, $\ldots$, $ k' \}$ to be equivalent if the components of $G$ and $G'$ are 
 the same, and we may identify $\pi$ with any graph $G$ equivalent to $\pi$. We may refer to $G$ or its equivalence class as a \emph{partition diagram}, 
 and it may equivalently be denoted as a set partition. 

\begin{example}
 The partition diagram corresponding to the set partition $\{\{5'$, $5\}$, $ \{4'\}$, $ \{3'$, $1$, $2$, $3$, $4\}$, $ \{2'\}$, $ \{1'\}\}$ may be 
 illustrated as below. 
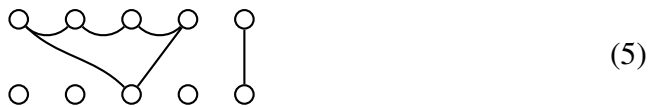
\begin{equation}\label{diagramgraphex}
 \begin{tikzpicture}[scale = 0.5,thick, baseline={(0,-1ex/2)}] 
\tikzstyle{vertex} = [shape = circle, minimum size = 7pt, inner sep = 1pt] 
\node[vertex] (G--5) at (6.0, -1) [shape = circle, draw] {}; 
\node[vertex] (G-5) at (6.0, 1) [shape = circle, draw] {}; 
\node[vertex] (G--4) at (4.5, -1) [shape = circle, draw] {}; 
\node[vertex] (G--3) at (3.0, -1) [shape = circle, draw] {}; 
\node[vertex] (G-1) at (0.0, 1) [shape = circle, draw] {}; 
\node[vertex] (G-2) at (1.5, 1) [shape = circle, draw] {}; 
\node[vertex] (G-3) at (3.0, 1) [shape = circle, draw] {}; 
\node[vertex] (G-4) at (4.5, 1) [shape = circle, draw] {}; 
\node[vertex] (G--2) at (1.5, -1) [shape = circle, draw] {}; 
\node[vertex] (G--1) at (0.0, -1) [shape = circle, draw] {}; 
\draw[] (G-5) .. controls +(0, -1) and +(0, 1) .. (G--5); 
\draw[] (G-1) .. controls +(0.5, -0.5) and +(-0.5, -0.5) .. (G-2); 
\draw[] (G-2) .. controls +(0.5, -0.5) and +(-0.5, -0.5) .. (G-3); 
\draw[] (G-3) .. controls +(0.5, -0.5) and +(-0.5, -0.5) .. (G-4); 
\draw[] (G-4) .. controls +(-0.75, -1) and +(0.75, 1) .. (G--3); 
\draw[] (G--3) .. controls +(-1, 1) and +(1, -1) .. (G-1); 
\end{tikzpicture} 
\end{equation}
\end{example}

 The set of all partition diagrams of order $k$ may be denoted as $A_{k}$, and this is typically endowed with a monoid structure. The partition algebra $ 
 \mathbb{C}A_k(n)$ is closely related to the monoid $A_{k}$, and we refer to \cite{HalversonRam2005} for background material on and definitions 
 concerning $\mathbb{C}A_k(n)$. The algebra $\mathbb{C}A_k(n)$ is 
 spanned by a 
 {diagram basis} given by the underlying set of $A_{k}$, so that the dimension of $\mathbb{C}A_{k}(n)$ is equal to $B_{2k}$, i.e., the Bell number 
 indexed by $2k$, where the Bell number $B_{m}$ is equal to the number of set partitions of $\{ 1, 2, \ldots, m \}$. 

 The multiplicative operations for both $A_{k}$ and 
 $\mathbb{C}A_k(n)$ are defined using the \emph{vertical} concatenation $d_{1} \ast d_{2}$ of partition diagrams, 
 as opposed to the horizontal concatenation operation $\otimes$ referred to as above. 
 The product $d_{1} \ast d_{2}$ 
 is obtained by placing $d_{1}$ on top of $d_{2}$ in such a way so that the bottom nodes of $d_{1}$ overlap with the top nodes of 
 $d_{2}$. The underlying multiplicative operation for $\mathbb{C}A_k(n)$
 may be defined by 
 removing the middle row of 
 $d_{1} \ast d_{2}$, in a way that preserves the relation of topmost nodes
 in $d_{1} \ast d_{2}$ being in the same component as bottommost nodes, 
 so that $n^{\ell}$ times the resultant diagram 
 is equal to the product $d_{1} d_{2}$ in $\mathbb{C}A_k(n)$, where $\ell$ denotes the number of components removed from the middle row 
 of $d_{1} \ast d_{2}$, and where $n$ denotes a complex parameter. 
 To form a graded algebra or graded ring structure based on a multiplicative 
 operation on partition diagrams, it would be appropriate to use the {horizontal} concatenation of 
 partition diagrams. Following Halverson's work in \cite{Halverson2001}, by letting $d_{1}$ and $d_{2}$ denote partition diagrams that are, respectively, 
 of orders $k_{1}$ and $k_{2}$, we may let $d_{1} \otimes d_{2}$ denote the partition diagram of order $k_{1} + k_{2}$ obtained by placing $d_{2}$ 
 to the right of $d_{1}$. 

 The partition algebra $\mathbb{C}A_k(n)$ naturally arises in the field of statistical mechanics via the Schur--Weyl duality associated with the 
 centralizer algebra 
\begin{equation}\label{maincentralizer}
 \text{End}_{S_{n}}\left( V^{\otimes k} \right) \cong \mathbb{C}A_k(n), 
\end{equation}
 for an $n$-dimensional vector space $V$, and where $S_{n}$ acts on $ V^{\otimes k} $ diagonally, as a subgroup of $\text{GL}_{n}(\mathbb{C})$. 
 The interdisciplinary interest surrounding the construction of a Hopf algebra on the combinatorial objects indexing the bases of centralizer algebras of the 
 form $\text{End}_{S_{n}}\left( V^{\otimes k} \right) $ is motivated by past research on Schur--Weyl duality and Hopf algebras, as in the work of 
 Benkart and Witherspoon \cite{BenkartWitherspoon2004} and the work of Novelli, Patras, and Thibon \cite{NovelliPatrasThibon2013}. 

\subsection{Relationship with $\textsf{NCSym}$}\label{subsectionNCSym}
 It appears that no coproduct operation has previously been defined on partition algebras or partition diagrams. However, there have been a number of 
 previously introduced Hopf algebras with bases indexed by set partitions or closely related combinatorial objects, which further motivates the interest in the 
 new Hopf algebra $\textsf{ParSym}$. What is typically meant by the Hopf algebra on set partitions refers to $\textsf{NCSym}$ (cf. \cite{Rey2007}), the 
 Hopf structure for which was introduced by Bergeron et al.\ in \cite{BergeronReutenauerRosasZabrocki2008}. One might think that the Hopf algebras $ 
 \textsf{NCSym}$ and $\textsf{ParSym}$ would be closely related, since both of these Hopf algebras have bases indexed by families of set partitions 
 and since both of these Hopf algebras contain $\textsf{NSym}$ as a Hopf subalgebra. However, our methods and results and constructions differ greatly 
 compared to \cite{BergeronReutenauerRosasZabrocki2008}, and this is summarized below. 

 Our CHA $\textsf{ParSym}$ naturally projects, via a CHA morphism, onto $\textsf{NSym}$, but $\textsf{NCSym}$ does not seem to project onto $ 
 \textsf{NSym}$, at least in any natural or useful or combinatorially significant way, 
 with reference 
 to \cite{BergeronReutenauerRosasZabrocki2008,BessenrodtLuotovanWilligenburg2011}, and with a particular regard toward 
 Theorem 4.9 in \cite{BergeronReutenauerRosasZabrocki2008}. The coproduct operation on $\textsf{NCSym}$ introduced in 
 \cite{BergeronReutenauerRosasZabrocki2008} is completely different from the coproduct for $\textsf{ParSym}$ that we introduce, which, arguably, gives us 
 a more natural lifting of $\textsf{NSym}$ in terms of how $\textsf{ParSym} $ projects onto $\textsf{NSym}$, compared to the comultiplication operation 
\begin{equation}\label{NCSymDelta}
 \Delta\left( \text{{\bf m}}_{A} \right) 
 = \sum_{S \subseteq [\ell(A)]} \text{{\bf m}}_{A_{S}} \otimes \text{{\bf m}}_{A_{S^{C}}}, 
\end{equation}
 referring to \cite{BergeronReutenauerRosasZabrocki2008} 
 for details as to the notation in \eqref{NCSymDelta}. 

 The bases of $\textsf{ParSym}$ are indexed by families of simple graphs denoting partition diagrams and indexing the bases for partition algebras, and 
 our work is directly motivated by combinatorial and representation-theoretic properties associated with partition diagrams and partition algebras, whereas 
 partition algebras, partition diagrams, graphs, etc., are not considered in \cite{BergeronReutenauerRosasZabrocki2008}. 
 Our explicit evaluation of the antipode $S_{\textsf{ParSym}}\colon \textsf{ParSym} \to \textsf{ParSym}$ 
 that we introduce requires the construction of a new, elementary-like basis, but the construction of such a basis is not 
 required in \cite{BergeronReutenauerRosasZabrocki2008}, and the antipode for $\textsf{NCSym}$ is not given explicitly in 
 \cite{BergeronReutenauerRosasZabrocki2008}. Finally, our lifting 
 properties associated with integer compositions from $\textsf{NSym}$ so as to be applicable in $\textsf{ParSym}$ requires the concept of a 
 $\bullet$-decomposition given in this article, 
 using an analogue, for partition diagrams, of near-concatenation, 
 but this kind of approach is not involved in \cite{BergeronReutenauerRosasZabrocki2008}. 

 The Aguiar--Orellana Hopf algebra on uniform block permutations \cite{AguiarOrellana2008} contains $\textsf{NCSym}$
 as a Hopf algebra, so much of the above commentary contrasting $\textsf{ParSym}$ and $\textsf{NCSym}$
 similarly applies with regard to the relationship between $\textsf{ParSym}$ and the Hopf algebra from 
 \cite{AguiarOrellana2008}. 
 We encourage the interested reader to consider past references that have been influenced by or otherwise reference 
 \cite{BergeronReutenauerRosasZabrocki2008} and that also motivate 
 our interest in $\textsf{ParSym}$, including \cite{offthecharts,AliniaeifardLivanWilligenburg2022,BergeronHohlwegRosasZabrocki2006,BergeronThiem2013,BessenrodtLuotovanWilligenburg2011,Thiem2010}. 
 
\section{Irreducible partition diagrams}\label{sectionnewHopf}
 We define 
\begin{equation}\label{gradedcomponent}
 \textsf{ParSym}_{i} = \text{span}_{\mathbbm{k}}\{ \textsf{H}_{\pi} : \pi \in A_{i} \}, 
\end{equation}
 for $i \in \mathbb{N}_{0}$, where an expression of the form $\textsf{H}_{\pi}$, 
 for a partition diagram $\pi$ that is irreducible according to Definition \ref{oirreducible} 
 below, may be seen as a variable, 
 by analogy with the $H$-generators for $\textsf{NSym}$. We adopt the convention whereby 
 $A_{0}$ consists of the unique ``empty partition diagram'' $\varnothing$ without any nodes. By direct analogy with \eqref{homogeneousmult}, 
 we define 
\begin{equation}\label{productParSym}
 \textsf{H}_{\pi} \textsf{H}_{\rho} = \textsf{H}_{\pi \otimes \rho} 
\end{equation}
 for partition diagrams $\pi$ and $\rho$. 

 We define 
\begin{equation}\label{ParSymdefinition}
 \textsf{ParSym} := \bigoplus_{i \in \mathbb{N}_{0}} \textsf{ParSym}_{i}, 
\end{equation}
 and we endow the direct sum of $\mathbbm{k}$-spans in 
 \eqref{ParSymdefinition} with the operation defined in \eqref{productParSym}, extended linearly, 
 yielding an associative operation on $\textsf{ParSym}$. 
 We let the morphism $\eta\colon \mathbbm{k} \to \textsf{ParSym}$ be such that 
\begin{equation}\label{etaParSym}
 \eta\left( 1_{\mathbbm{k}} \right) = \textsf{H}_{\varnothing}, 
\end{equation}
 giving a unit morphism 
 that gives $\textsf{ParSym}$ the structure of an associative $\mathbbm{k}$-algebra. 
 
\begin{definition}\label{oirreducible}
 For a partition diagram $\pi$, let $\pi$ be referred to as being \emph{$\otimes$-irreducible} if it cannot be expressed as $\rho^{(1)} \otimes 
 \rho^{(2)}$ for nonempty partition diagrams $\rho^{(1)}$ and $\rho^{(2)}$. 
\end{definition}

 We are to later refer to the concept of $\bullet$-irreducibility, in contrast to Definition \ref{oirreducible}, for an analogue $\bullet$ of near-concatenations 
 for integer compositions. For the sake of convenience, we may refer to $\otimes$-irreducibility as irreducibility, depending on the context. 

\begin{example}\label{irreducible2}
 We may verify that the irreducible diagrams in $\textsf{ParSym}_{2}$ are as below. 
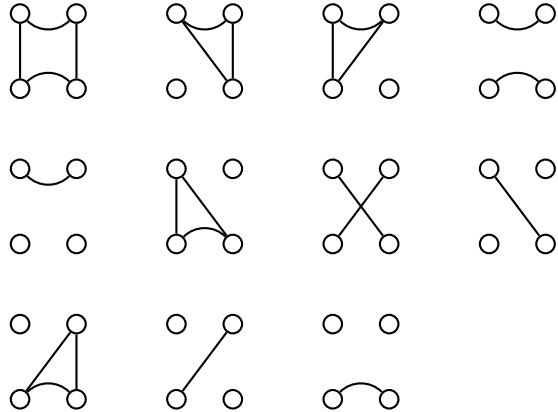
\begin{align*}
 & \begin{tikzpicture}[scale = 0.5,thick, baseline={(0,-1ex/2)}] 
\tikzstyle{vertex} = [shape = circle, minimum size = 7pt, inner sep = 1pt] 
\node[vertex] (G--2) at (1.5, -1) [shape = circle, draw] {}; 
\node[vertex] (G--1) at (0.0, -1) [shape = circle, draw] {}; 
\node[vertex] (G-1) at (0.0, 1) [shape = circle, draw] {}; 
\node[vertex] (G-2) at (1.5, 1) [shape = circle, draw] {}; 
\draw[] (G-1) .. controls +(0.5, -0.5) and +(-0.5, -0.5) .. (G-2); 
\draw[] (G-2) .. controls +(0, -1) and +(0, 1) .. (G--2); 
\draw[] (G--2) .. controls +(-0.5, 0.5) and +(0.5, 0.5) .. (G--1); 
\draw[] (G--1) .. controls +(0, 1) and +(0, -1) .. (G-1); 
\end{tikzpicture} \ \ \ \ \ \ \ \ \ \
 \begin{tikzpicture}[scale = 0.5,thick, baseline={(0,-1ex/2)}] 
\tikzstyle{vertex} = [shape = circle, minimum size = 7pt, inner sep = 1pt] 
\node[vertex] (G--2) at (1.5, -1) [shape = circle, draw] {}; 
\node[vertex] (G-1) at (0.0, 1) [shape = circle, draw] {}; 
\node[vertex] (G-2) at (1.5, 1) [shape = circle, draw] {}; 
\node[vertex] (G--1) at (0.0, -1) [shape = circle, draw] {}; 
\draw[] (G-1) .. controls +(0.5, -0.5) and +(-0.5, -0.5) .. (G-2); 
\draw[] (G-2) .. controls +(0, -1) and +(0, 1) .. (G--2); 
\draw[] (G--2) .. controls +(-0.75, 1) and +(0.75, -1) .. (G-1); 
\end{tikzpicture} \ \ \ \ \ \ \ \ \ \ 
\begin{tikzpicture}[scale = 0.5,thick, baseline={(0,-1ex/2)}] 
\tikzstyle{vertex} = [shape = circle, minimum size = 7pt, inner sep = 1pt] 
\node[vertex] (G--2) at (1.5, -1) [shape = circle, draw] {}; 
\node[vertex] (G--1) at (0.0, -1) [shape = circle, draw] {}; 
\node[vertex] (G-1) at (0.0, 1) [shape = circle, draw] {}; 
\node[vertex] (G-2) at (1.5, 1) [shape = circle, draw] {}; 
\draw[] (G-1) .. controls +(0.5, -0.5) and +(-0.5, -0.5) .. (G-2); 
\draw[] (G-2) .. controls +(-0.75, -1) and +(0.75, 1) .. (G--1); 
\draw[] (G--1) .. controls +(0, 1) and +(0, -1) .. (G-1); 
\end{tikzpicture} \ \ \ \ \ \ \ \ \ \ 
 \begin{tikzpicture}[scale = 0.5,thick, baseline={(0,-1ex/2)}] 
\tikzstyle{vertex} = [shape = circle, minimum size = 7pt, inner sep = 1pt] 
\node[vertex] (G--2) at (1.5, -1) [shape = circle, draw] {}; 
\node[vertex] (G--1) at (0.0, -1) [shape = circle, draw] {}; 
\node[vertex] (G-1) at (0.0, 1) [shape = circle, draw] {}; 
\node[vertex] (G-2) at (1.5, 1) [shape = circle, draw] {}; 
\draw[] (G--2) .. controls +(-0.5, 0.5) and +(0.5, 0.5) .. (G--1); 
\draw[] (G-1) .. controls +(0.5, -0.5) and +(-0.5, -0.5) .. (G-2); 
\end{tikzpicture} \\ 
 & \ \\ 
 & \begin{tikzpicture}[scale = 0.5,thick, baseline={(0,-1ex/2)}] 
\tikzstyle{vertex} = [shape = circle, minimum size = 7pt, inner sep = 1pt] 
\node[vertex] (G--2) at (1.5, -1) [shape = circle, draw] {}; 
\node[vertex] (G--1) at (0.0, -1) [shape = circle, draw] {}; 
\node[vertex] (G-1) at (0.0, 1) [shape = circle, draw] {}; 
\node[vertex] (G-2) at (1.5, 1) [shape = circle, draw] {}; 
\draw[] (G-1) .. controls +(0.5, -0.5) and +(-0.5, -0.5) .. (G-2); 
\end{tikzpicture} \ \ \ \ \ \ \ \ \ \ 
 \begin{tikzpicture}[scale = 0.5,thick, baseline={(0,-1ex/2)}] 
\tikzstyle{vertex} = [shape = circle, minimum size = 7pt, inner sep = 1pt] 
\node[vertex] (G--2) at (1.5, -1) [shape = circle, draw] {}; 
\node[vertex] (G--1) at (0.0, -1) [shape = circle, draw] {}; 
\node[vertex] (G-1) at (0.0, 1) [shape = circle, draw] {}; 
\node[vertex] (G-2) at (1.5, 1) [shape = circle, draw] {}; 
\draw[] (G-1) .. controls +(0.75, -1) and +(-0.75, 1) .. (G--2); 
\draw[] (G--2) .. controls +(-0.5, 0.5) and +(0.5, 0.5) .. (G--1); 
\draw[] (G--1) .. controls +(0, 1) and +(0, -1) .. (G-1); 
\end{tikzpicture} \ \ \ \ \ \ \ \ \ \ 
 \begin{tikzpicture}[scale = 0.5,thick, baseline={(0,-1ex/2)}] 
\tikzstyle{vertex} = [shape = circle, minimum size = 7pt, inner sep = 1pt] 
\node[vertex] (G--2) at (1.5, -1) [shape = circle, draw] {}; 
\node[vertex] (G-1) at (0.0, 1) [shape = circle, draw] {}; 
\node[vertex] (G--1) at (0.0, -1) [shape = circle, draw] {}; 
\node[vertex] (G-2) at (1.5, 1) [shape = circle, draw] {}; 
\draw[] (G-1) .. controls +(0.75, -1) and +(-0.75, 1) .. (G--2); 
\draw[] (G-2) .. controls +(-0.75, -1) and +(0.75, 1) .. (G--1); 
\end{tikzpicture} \ \ \ \ \ \ \ \ \ \ 
 \begin{tikzpicture}[scale = 0.5,thick, baseline={(0,-1ex/2)}] 
\tikzstyle{vertex} = [shape = circle, minimum size = 7pt, inner sep = 1pt] 
\node[vertex] (G--2) at (1.5, -1) [shape = circle, draw] {}; 
\node[vertex] (G-1) at (0.0, 1) [shape = circle, draw] {}; 
\node[vertex] (G--1) at (0.0, -1) [shape = circle, draw] {}; 
\node[vertex] (G-2) at (1.5, 1) [shape = circle, draw] {}; 
\draw[] (G-1) .. controls +(0.75, -1) and +(-0.75, 1) .. (G--2); 
\end{tikzpicture} \\
 & \ \\
 & \begin{tikzpicture}[scale = 0.5,thick, baseline={(0,-1ex/2)}] 
\tikzstyle{vertex} = [shape = circle, minimum size = 7pt, inner sep = 1pt] 
\node[vertex] (G--2) at (1.5, -1) [shape = circle, draw] {}; 
\node[vertex] (G--1) at (0.0, -1) [shape = circle, draw] {}; 
\node[vertex] (G-2) at (1.5, 1) [shape = circle, draw] {}; 
\node[vertex] (G-1) at (0.0, 1) [shape = circle, draw] {}; 
\draw[] (G-2) .. controls +(0, -1) and +(0, 1) .. (G--2); 
\draw[] (G--2) .. controls +(-0.5, 0.5) and +(0.5, 0.5) .. (G--1); 
\draw[] (G--1) .. controls +(0.75, 1) and +(-0.75, -1) .. (G-2); 
\end{tikzpicture} \ \ \ \ \ \ \ \ \ \ 
 \begin{tikzpicture}[scale = 0.5,thick, baseline={(0,-1ex/2)}] 
\tikzstyle{vertex} = [shape = circle, minimum size = 7pt, inner sep = 1pt] 
\node[vertex] (G--2) at (1.5, -1) [shape = circle, draw] {}; 
\node[vertex] (G--1) at (0.0, -1) [shape = circle, draw] {}; 
\node[vertex] (G-2) at (1.5, 1) [shape = circle, draw] {}; 
\node[vertex] (G-1) at (0.0, 1) [shape = circle, draw] {}; 
\draw[] (G-2) .. controls +(-0.75, -1) and +(0.75, 1) .. (G--1); 
\end{tikzpicture} \ \ \ \ \ \ \ \ \ \ 
 \begin{tikzpicture}[scale = 0.5,thick, baseline={(0,-1ex/2)}] 
\tikzstyle{vertex} = [shape = circle, minimum size = 7pt, inner sep = 1pt] 
\node[vertex] (G--2) at (1.5, -1) [shape = circle, draw] {}; 
\node[vertex] (G--1) at (0.0, -1) [shape = circle, draw] {}; 
\node[vertex] (G-1) at (0.0, 1) [shape = circle, draw] {}; 
\node[vertex] (G-2) at (1.5, 1) [shape = circle, draw] {}; 
\draw[] (G--2) .. controls +(-0.5, 0.5) and +(0.5, 0.5) .. (G--1); 
\end{tikzpicture} \ \ \ \ \ \ \ \ \ \ 
\end{align*}
\end{example}

\begin{remark}\label{realizedthiswas}
 We adopt the convention whereby a given partition diagram denoted as a graph is written in such a way so that any of its edges are within the 
 rectangular formation, including the borders, given by the upper and lower nodes of the graph, 
 which are arranged horizontally, as in the partition diagram illustrations we have 
 previously provided. 
\end{remark}

 Let $\pi$ be a partition diagram in $\textsf{ParSym}_{k}$. 
 We find that $\pi$ 
 may be expressed as the concatenation of 
 two nonempty diagrams if and only if there is a natural number $j \in \{ 1, 2, \ldots, k - 1 \}$ 
 such that $\pi$ may be drawn in such a way so that no edge of $\pi$ crosses 
 a vertical line that is placed between $j$ and $j + 1$ and between $j'$ and $(j+1)'$. 
 We refer to a vertical line of this form as a \emph{separating line}, 
 and we may refer to there being a \emph{separation} between $j$ and $j+1$ and between 
 $j'$ and $(j+1)'$ in $\pi$ if the situation described in the preceding sentence holds. 

\begin{lemma}\label{lemmauniqueotimes}
 Let $\pi$ be a partition diagram. Then $\pi$ can be uniquely written in the form 
 $$ \pi = \pi^{(1)} \otimes \pi^{(2)} \otimes \cdots \otimes \pi^{(n)}, $$
 with $n \in \mathbb{N}$, and where $\pi^{(1)}$, $\pi^{(2)}$, $\ldots$, $\pi^{(n)}$ 
 are $\otimes$-irreducible diagrams. 
\end{lemma}

\begin{proof}
 A partition diagram is non-irreducible if and only if it has at least one separation. 
 By taking all possible separations for a given partition diagram, 
 this leads us to construct a bijection between 
 non-irreducible partition diagrams $\rho$ of order $k$ and ordered tuples of the form 
\begin{equation}\label{orderedform}
 \left( \pi^{\alpha_{1}}, \pi^{\alpha_{2}}, \ldots, \pi^{\alpha_{\ell(\alpha)}} \right), 
\end{equation}
 where $\pi^{\alpha_{i}}$, for a given index $i$, is an irreducible partition diagram 
 of order $\alpha_{i}$, and where $\alpha = (\alpha_{1}, \alpha_{2}, \ldots, \alpha_{\ell(\alpha)})$
 is an integer composition of $k$ such that $\alpha \neq (k)$. Explicitly, 
 the separations for a given diagram $\rho$
 partition $\rho$ in such a way so that 
\begin{equation}\label{takeallforrho}
 \rho = \pi^{\alpha_{1}} \otimes \pi^{\alpha_{2}} \otimes \cdots \otimes \pi^{\alpha_{\ell(\alpha)}}, 
\end{equation}
 so we let $\rho$ be mapped to \eqref{orderedform}, 
 with the surjectivity being immediate from the condition that $\alpha \neq (k)$. 
 For domain elements 
 $ \rho^{(1)} = \pi^{\alpha_{1}} \otimes \pi^{\alpha_{2}} \otimes \cdots 
 \otimes \pi^{\alpha_{\ell(\alpha)}}$ 
 and $ \rho^{(2)} = \mu^{\beta_{1}} \otimes \mu^{\beta_{2}} 
 \otimes \cdots \otimes \mu^{\beta_{\ell(\beta)}}$ written as $\otimes$-products of nonempty, irreducible diagrams, with at least two 
 factors in each case, 
 the equality of the images of these domain elements
 gives us that $\rho^{(1)} = \rho^{(2)}$ in an immediate fashion, 
 by comparing the lengths and entries of the corresponding tuples 
 and by appealing to the irreducibility of the entries. 
\end{proof}

\begin{theorem}\label{recursiona}
 Let $a_{k}$ denote the number of irreducible diagrams in $A_{k} $
 for a positive integer $k$. Then the recursion 
\begin{equation}\label{290923808582878284858AM1A}
 a_{k} = B_{2 k} - \sum_{\substack{\alpha \vDash k \\ \alpha \neq (k) }} 
 a_{\alpha_{1}} a_{\alpha_{2}} \cdots a_{\alpha_{\ell(\alpha)}} 
\end{equation}
 holds, with $a_{1} = 2$. 
\end{theorem}

\begin{proof}
 By Lemma \ref{lemmauniqueotimes}, the number of irreducible diagrams in $A_{k}$ is equal to $B_{2k}$ minus the cardinality of the set of all possible 
 tuples that are of the form indicated in \eqref{orderedform} and that are subject to the given conditions. 
 This gives us an equivalent version of \eqref{290923808582878284858AM1A}. 
\end{proof}

\begin{example}
 According to the recursion and the initial condition specified in Theorem \ref{recursiona}, 
 we find that 
\begin{equation}\label{notOEIS}
 (a_{k} : k \in \mathbb{N}) = (2, 11, 151, 3267, 96663, 3663123, 171131871, \ldots), 
\end{equation}
 noting that the diagrams corresponding to the $a_{2} = 11$ evaluation are shown in Example \ref{irreducible2}. 
 The integer sequence in \eqref{notOEIS}
 is not currently included in the On-Line Encyclopedia of Integer Sequences, 
 and neither of the integers among 3663123 and 
 171131871 shown in \eqref{notOEIS} are currently in the OEIS. 
 This strongly suggests that 
 the free algebra structure indicated in Theorem \ref{2qqqqq0q2q3q0q7q1q610AAA4akm1a} is original. 
\end{example}

 A \emph{free $\mathbbm{k}$-algebra} 
 is of the form 
\begin{equation}\label{freedefinition}
 \mathbbm{k}\langle X \rangle = \bigoplus_{w \in X^{\ast}} \mathbbm{k}w, 
\end{equation}
 where the multiplicative operation is given by concatenation of words in the free monoid $X^{\ast}$, and where this operation is extended linearly, and 
 where $\mathbbm{k}w$ denotes the free $\mathbbm{k}$-module on the singleton set $\{ w \}$. 

\begin{theorem}\label{2qqqqq0q2q3q0q7q1q610AAA4akm1a}
 As a $\mathbbm{k}$-algebra, $\textsf{\emph{ParSym}}$ is the free $\mathbbm{k}$-algebra
 with $a_{k}$ generators in each degree, for $a_{k}$ as in Theorem \ref{recursiona}. 
\end{theorem}

\begin{proof}
 From \eqref{ParSymdefinition}, we may obtain an algebra isomorphism with an algebra of the form indicated in \eqref{freedefinition}, by expressing the 
 basis elements in the graded components shown in \eqref{ParSymdefinition} using irreducible diagrams, in the following manner. 
 According to Lemma \ref{lemmauniqueotimes}, 
 for a partition diagram $\rho$, by taking all possible separating lines for $\rho$, this gives us how $\rho$ may be expressed in a 
 unique way as a concatenation of irreducible diagrams. With respect to the notation in \eqref{freedefinition}, we set $X$ to be the set of all irreducible 
 diagrams. So, by identifying the concatenation in \eqref{takeallforrho} with the word $ \pi^{\alpha_{1}} \pi^{\alpha_{2}} \cdots 
 \pi^{\alpha_{\ell(\alpha)}}$ in $X^{\ast}$, this gives us an algebra isomorphism of the desired form. 
\end{proof}

\begin{definition}
 For a sequence $(\textsf{a}_{n} : n \in \mathbb{N})$, the \emph{Boolean transform} $(\textsf{b}_{n} : n \in \mathbb{N})$ for the $\textsf{a}$-sequence 
 may be defined so that (cf.\ \cite{AguiarMahajan2014}) 
\begin{equation}\label{Booleangf}
 \sum_{n=1}^{\infty} \textsf{b}_{n} x^{n} := 1 - \frac{1}{1 + \sum_{n=1}^{\infty} \textsf{a}_{n} x^{n}}. 
\end{equation}
\end{definition}

 From the definition in \eqref{Booleangf}, a combinatorial argument involving the Cauchy product of the generating functions involved can be used to prove 
 the equivalent definition of the Boolean transform indicated below \cite{AguiarMahajan2014}: 
\begin{equation}\label{equivalentBoolean}
 \textsf{a}_{n} = \sum_{\alpha \vDash n} 
 \textsf{b}_{\alpha_{1}} \textsf{b}_{\alpha_{2}} \cdots \textsf{b}_{\alpha_{\ell(\alpha)}}. 
\end{equation}
 So, by rewriting \eqref{equivalentBoolean} so that 
\begin{equation}\label{sfbsfa}
 \textsf{b}_{n} = \textsf{a}_{n} 
 - \sum_{\substack{ \alpha \vDash n \\ \alpha \neq (n) }} 
 \textsf{b}_{\alpha_{1}} \textsf{b}_{\alpha_{2}} \cdots \textsf{b}_{\alpha_{\ell(\alpha)}}, 
\end{equation}
 we find that the recursion in Theorem \ref{recursiona} is of the form indicated in \eqref{sfbsfa}, leading us toward the generating function highlighted in 
 Theorem \ref{theoremgf} below. This gives us an analogue for partition diagrams of Comtet's formula for enumerating 
 irreducible permutations \cite{Comtet1972} (cf.\ \cite{GaoKitaevZhang2018,King2006}). 

\begin{theorem}\label{theoremgf}
 The generating function for the sequence $(a_{k} : k \in \mathbb{N})$ given in Theorem \ref{recursiona} satisfies $$ \sum_{k=1}^{\infty} a_{k} 
 x^{k} = 1 - \frac{1}{1 + \sum_{k=1}^{\infty} B_{2k} x^{k}}. $$ 
\end{theorem}

\begin{proof}
 Since $(a_{k} : k \in \mathbb{N})$ is the Boolean transform of $(B_{2k} : k \in \mathbb{N})$, 
 the desired generating function identity is immediate from the equivalence of 
 \eqref{Booleangf} and \eqref{sfbsfa}. 
\end{proof}

\section{A CHA on partition diagrams}\label{sectionCHA}
 For a coalgebra $C$ over $\mathbbm{k}$, the coproduct morphism $\Delta\colon C \to C \otimes C$ satisfies the coassociativity axiom, and for a 
 bialgebra, the operation $\Delta$ would be compatible with the multiplicative operation. So, to endow $\textsf{ParSym}$ with a bialgebra structure, 
 we would need to define a coproduct operation $\Delta$ such that 
\begin{equation}\label{Deltacompatible}
 \Delta(\textsf{H}_{\pi} \textsf{H}_{\rho}) = \Delta(\textsf{H}_{\pi}) \Delta(\textsf{H}_{\rho})
\end{equation}
 for irreducible generators $\textsf{H}_{\pi}$ and $\textsf{H}_{\rho}$ in $\textsf{ParSym}$, with 
\begin{align*}
 \Delta\left( \textsf{H}_{\pi^{(1)} \otimes \pi^{(2)} \otimes \cdots \otimes \pi^{(\ell)} } \right) 
 & = \Delta\left( \textsf{H}_{\pi^{(1)} } \textsf{H}_{\pi^{(2)} } \cdots \textsf{H}_{\pi^{(\ell)} } \right) \\ 
 & = \Delta\left( \textsf{H}_{\pi^{(1)} } \right)
 \Delta\left( \textsf{H}_{\pi^{(2)} } \right) \cdots \Delta\left( \textsf{H}_{\pi^{(\ell)} } \right) 
\end{align*}
 for $\otimes$-irreducible partition diagrams $\pi^{(1)}$, $\pi^{(2)}$, $\ldots$, $\pi^{(\ell)}$. 
 This leads us to consider what would be appropriate as a way of lifting the 
 expansion formula 
\begin{equation}\label{DeltaNSym}
 \Delta\left( H_{n} \right) = H_{0} \otimes H_{n} 
 + H_{1} \otimes H_{n-1} + \cdots + H_{n} \otimes H_{0} 
\end{equation}
 for coproducts of generators in $\textsf{NSym}$. Since the generators of $\textsf{ParSym}$ are indexed by graphs, this leads us to consider how 
 the expansion in \eqref{DeltaNSym} could be reformulated and lifted in a graph-theoretic way. 

 The generator $H_{n}$ may be rewritten so as to be indexed by an integer composition, in accordance with \eqref{completefamily}, writing $H_{n} = 
 H_{(n)}$. By denoting compositions as composition tableaux, we may let this generator be denoted as 
\begin{equation}\label{Hyoung}
 H_{\underbrace{\young(\null\null\cdots\null)}_{n}}. 
\end{equation}
 We may rewrite the index in \eqref{Hyoung} as a path graph with $n$ nodes, by analogy with Ferrers diagrams and by analogy with 
 our notation for partition diagrams as in \eqref{diagramgraphex}: 
\begin{equation}\label{Hpath}
 H_{ \begin{tikzpicture}[scale = 0.5,thick, baseline={(0,-1ex/2)}] 
\tikzstyle{vertex} = [shape = circle, minimum size = 7pt, inner sep = 1pt] 
\node[vertex] (G-1) at (0.0, 0) [shape = circle, draw] {}; 
\node[vertex] (G-2) at (1.5, 0) [shape = circle, draw] {}; 
\node[vertex] (G-3) at (3.0, 0) [shape = circle, draw] {}; 
\draw[] (G-1) .. controls +(0.4, -0.0) and +(-0.4, -0.0) .. (G-2); 
\draw[] (G-2) .. controls +(0.4, -0.0) and +(-0.4, -0.0) .. (G-3); 
\end{tikzpicture} \ \cdots \ 
 \begin{tikzpicture}[scale = 0.5,thick, baseline={(0,-1ex/2)}] 
\tikzstyle{vertex} = [shape = circle, minimum size = 7pt, inner sep = 1pt] 
\node[vertex] (G-1) at (0.0, 0) [shape = circle, draw] {}; 
\node[vertex] (G-2) at (1.5, 0) [shape = circle, draw] {}; 
\draw[] (G-1) .. controls +(0.4, -0.0) and +(-0.4, -0.0) .. (G-2); 
\end{tikzpicture}}. 
\end{equation}

\subsection{Near-concatenations of partition diagrams}\label{subsectionNear}
 The \emph{near-concatenation} of compositions $\alpha$ and $\beta$, which we may denote as $\alpha \bullet \beta$, is given by the composition 
 $(\alpha_{1}$, $\alpha_{2}$, $ \ldots$, $ \alpha_{\ell(\alpha) - 1}$, $ \alpha_{\ell(\alpha)} $ $ + $ $ \beta_{1}$, 
 $ \beta_{2}$, $ \beta_{3}$, $ \ldots$, $ \beta_{\ell(\beta)} )$
 and naturally arises in the context of noncommutative symmetric functions, as 
 in with the multiplication rule for the ribbon basis of $\textsf{NSym}$. 
 Grinberg \cite{Grinberg2017} has explored the free algebra structure given by endowing the dual of $\textsf{NSym}$ 
 with an operation that may be defined via the near-concatenation
 of compositions, and this has inspired us to 
 reformulate $\textsf{NSym}$ in a related way in terms of $\bullet$. 

 According to the notation in \eqref{Hpath}, the binary operation $\bullet$ has the effect of joining concatenated path graphs to form a path graph, 
 by adding an edge appropriately. If we take the singleton set consisting of $H_{(1)}$, we may redefine $\textsf{NSym}$ as the underlying algebra of a 
 free structure generated by this singleton set 
    with the multiplicative operations     $\bullet$   and $\circ$, 
  where $\circ$ denotes the usual  multiplication operation for $\textsf{NSym}$. 

\begin{example}
 We may rewrite the complete homogeneous basis element $ H_{(3, 1, 4)}$ as 
\begin{equation}\label{bulletwithcirc}
 H_{1} \bullet H_{1} \bullet H_{1} \circ H_{1} \circ H_{1} \bullet H_{1} \bullet H_{1} \bullet H_{1}
\end{equation}
 or as 
\begin{equation}\label{path314}
 H_{ \begin{tikzpicture}[scale = 0.5,thick, baseline={(0,-1ex/2)}] 
\tikzstyle{vertex} = [shape = circle, minimum size = 7pt, inner sep = 1pt] 
\node[vertex] (G-1) at (0.0, 0) [shape = circle, draw] {}; 
\node[vertex] (G-2) at (1.5, 0) [shape = circle, draw] {}; 
\node[vertex] (G-3) at (3.0, 0) [shape = circle, draw] {}; 
\draw[] (G-1) .. controls +(0.4, -0.0) and +(-0.4, -0.0) .. (G-2); 
\draw[] (G-2) .. controls +(0.4, -0.0) and +(-0.4, -0.0) .. (G-3); 
\end{tikzpicture}} \ H_{ \begin{tikzpicture}[scale = 0.5,thick, baseline={(0,-1ex/2)}] 
\tikzstyle{vertex} = [shape = circle, minimum size = 7pt, inner sep = 1pt] 
\node[vertex] (G-1) at (0.0, 0) [shape = circle, draw] {}; 
\end{tikzpicture} 
 } \ H_{ \begin{tikzpicture}[scale = 0.5,thick, baseline={(0,-1ex/2)}] 
\tikzstyle{vertex} = [shape = circle, minimum size = 7pt, inner sep = 1pt] 
\node[vertex] (G-1) at (0.0, 0) [shape = circle, draw] {}; 
\node[vertex] (G-2) at (1.5, 0) [shape = circle, draw] {}; 
\node[vertex] (G-3) at (3.0, 0) [shape = circle, draw] {}; 
\node[vertex] (G-4) at (4.5, 0) [shape = circle, draw] {}; 
\draw[] (G-1) .. controls +(0.4, -0.0) and +(-0.4, -0.0) .. (G-2); 
\draw[] (G-2) .. controls +(0.4, -0.0) and +(-0.4, -0.0) .. (G-3); 
\draw[] (G-3) .. controls +(0.4, -0.0) and +(-0.4, -0.0) .. (G-4); 
\end{tikzpicture}}. 
\end{equation}
\end{example}

 By rewriting \eqref{DeltaNSym} as the sum of all expressions of the form $H_{i} \otimes H_{j}$ such that $H_{i} \bullet H_{j} = H_{n}$, this, together 
 with the graph-theoretic interpretation of the operation $\bullet$, leads us to consider a lifting of the coproduct of $\textsf{NSym}$ to $\textsf{ParSym}$. 

\begin{definition}\label{bulletdefinition}
 For nonempty partition diagrams $\pi$ and $\rho$, define the binary operation $\bullet$ so that 
 $ \textsf{H}_{\pi} \bullet \textsf{H}_{\rho} = \textsf{H}_{G} $ and 
 $ \pi \bullet \rho = G$, 
 where $G$ denotes the partition diagram equivalent to the graph $G$ obtained by joining the graphs of $\pi$ and $\rho$ with an edge incident with the 
 bottom-right node of $\pi$ and the bottom-left node of $\rho$. By convention, we set $\textsf{H}_{\varnothing} \bullet \textsf{H}_{\rho} = 
 \textsf{H}_{\rho}$ and $\textsf{H}_{\pi} \bullet \textsf{H}_{\varnothing} = \textsf{H}_{\pi}$, 
 with $\varnothing \bullet \rho = \rho$ and $\pi \bullet \varnothing = \pi$. 
\end{definition}

 Definition \ref{bulletdefinition} 
 gives us an analogue of near-concatenation for partition diagrams, 
 in view, for example, of the equivalence of \eqref{bulletwithcirc} and \eqref{path314}. 

\begin{example}
 According to the above definition of $\bullet$ for $\textsf{ParSym}$, we have that 
 $$ \textsf{H}_{ \begin{tikzpicture}[scale = 0.5,thick, baseline={(0,-1ex/2)}] 
\tikzstyle{vertex} = [shape = circle, minimum size = 7pt, inner sep = 1pt] 
\node[vertex] (G--3) at (3.0, -1) [shape = circle, draw] {}; 
\node[vertex] (G--2) at (1.5, -1) [shape = circle, draw] {}; 
\node[vertex] (G--1) at (0.0, -1) [shape = circle, draw] {}; 
\node[vertex] (G-3) at (3.0, 1) [shape = circle, draw] {}; 
\node[vertex] (G-1) at (0.0, 1) [shape = circle, draw] {}; 
\node[vertex] (G-2) at (1.5, 1) [shape = circle, draw] {}; 
\draw[] (G--3) .. controls +(-0.5, 0.5) and +(0.5, 0.5) .. (G--2); 
\draw[] (G-3) .. controls +(-1, -1) and +(1, 1) .. (G--1); 
\draw[] (G-1) .. controls +(0.5, -0.5) and +(-0.5, -0.5) .. (G-2); 
\end{tikzpicture}} \bullet \textsf{H}_{\begin{tikzpicture}[scale = 0.5,thick, baseline={(0,-1ex/2)}] 
\tikzstyle{vertex} = [shape = circle, minimum size = 7pt, inner sep = 1pt] 
\node[vertex] (G--4) at (4.5, -1) [shape = circle, draw] {}; 
\node[vertex] (G--3) at (3.0, -1) [shape = circle, draw] {}; 
\node[vertex] (G--1) at (0.0, -1) [shape = circle, draw] {}; 
\node[vertex] (G-1) at (0.0, 1) [shape = circle, draw] {}; 
\node[vertex] (G-2) at (1.5, 1) [shape = circle, draw] {}; 
\node[vertex] (G-3) at (3.0, 1) [shape = circle, draw] {}; 
\node[vertex] (G--2) at (1.5, -1) [shape = circle, draw] {}; 
\node[vertex] (G-4) at (4.5, 1) [shape = circle, draw] {}; 
\draw[] (G-1) .. controls +(0.5, -0.5) and +(-0.5, -0.5) .. (G-2); 
\draw[] (G-2) .. controls +(0.5, -0.5) and +(-0.5, -0.5) .. (G-3); 
\draw[] (G-3) .. controls +(0.75, -1) and +(-0.75, 1) .. (G--4); 
\draw[] (G--4) .. controls +(-0.5, 0.5) and +(0.5, 0.5) .. (G--3); 
\draw[] (G--3) .. controls +(-0.6, 0.6) and +(0.6, 0.6) .. (G--1); 
\draw[] (G--1) .. controls +(0, 1) and +(0, -1) .. (G-1); 
\draw[] (G-4) .. controls +(-1, -1) and +(1, 1) .. (G--2); 
\end{tikzpicture}} $$
 evaluates as 
 $$ \textsf{H}_{\begin{tikzpicture}[scale = 0.5,thick, baseline={(0,-1ex/2)}] 
\tikzstyle{vertex} = [shape = circle, minimum size = 7pt, inner sep = 1pt] 
\node[vertex] (G--7) at (9.0, -1) [shape = circle, draw] {}; 
\node[vertex] (G--6) at (7.5, -1) [shape = circle, draw] {}; 
\node[vertex] (G--4) at (4.5, -1) [shape = circle, draw] {}; 
\node[vertex] (G--3) at (3.0, -1) [shape = circle, draw] {}; 
\node[vertex] (G--2) at (1.5, -1) [shape = circle, draw] {}; 
\node[vertex] (G-4) at (4.5, 1) [shape = circle, draw] {}; 
\node[vertex] (G-5) at (6.0, 1) [shape = circle, draw] {}; 
\node[vertex] (G-6) at (7.5, 1) [shape = circle, draw] {}; 
\node[vertex] (G--5) at (6.0, -1) [shape = circle, draw] {}; 
\node[vertex] (G-7) at (9.0, 1) [shape = circle, draw] {}; 
\node[vertex] (G--1) at (0.0, -1) [shape = circle, draw] {}; 
\node[vertex] (G-3) at (3.0, 1) [shape = circle, draw] {}; 
\node[vertex] (G-1) at (0.0, 1) [shape = circle, draw] {}; 
\node[vertex] (G-2) at (1.5, 1) [shape = circle, draw] {}; 
\draw[] (G-4) .. controls +(0.5, -0.5) and +(-0.5, -0.5) .. (G-5); 
\draw[] (G-5) .. controls +(0.5, -0.5) and +(-0.5, -0.5) .. (G-6); 
\draw[] (G-6) .. controls +(0.75, -1) and +(-0.75, 1) .. (G--7); 
\draw[] (G--7) .. controls +(-0.5, 0.5) and +(0.5, 0.5) .. (G--6); 
\draw[] (G--6) .. controls +(-0.6, 0.6) and +(0.6, 0.6) .. (G--4); 
\draw[] (G--4) .. controls +(-0.5, 0.5) and +(0.5, 0.5) .. (G--3); 
\draw[] (G--3) .. controls +(-0.5, 0.5) and +(0.5, 0.5) .. (G--2); 
\draw[] (G--2) .. controls +(1, 1) and +(-1, -1) .. (G-4); 
\draw[] (G-7) .. controls +(-1, -1) and +(1, 1) .. (G--5); 
\draw[] (G-3) .. controls +(-1, -1) and +(1, 1) .. (G--1); 
\draw[] (G-1) .. controls +(0.5, -0.5) and +(-0.5, -0.5) .. (G-2); 
\end{tikzpicture}}. $$
\end{example}

\begin{definition}\label{DeltaParSym}
 For an irreducible partition diagram $\pi$, define the coproduct operation
 $\Delta$ on $\textsf{ParSym} $ so that 
\begin{equation}\label{displayDelta}
 \Delta \textsf{H}_{\pi} = \sum_{G_{1}, G_{2}} 
 \textsf{H}_{G_{1}} \otimes \textsf{H}_{G_{2}}, 
\end{equation}
 where the sum in \eqref{displayDelta} is over all $\otimes$-irreducible partitions diagrams $G_{1}$ and $G_{2}$ such that $ 
 \textsf{H}_{G_{1}} \bullet \textsf{H}_{G_{2}} = \textsf{H}_{\pi}$. We let $\Delta$ be compatible with the multiplicative operation of $\textsf{ParSym}$, 
 i.e., so that \eqref{Deltacompatible} holds, and we let $\Delta$ be extended linearly, so as to obtain a morphism $\Delta\colon \textsf{ParSym} \to 
 \textsf{ParSym} \otimes \textsf{ParSym}$. 
\end{definition}

\begin{example}\label{exampleDelta}
 According to the definition of $\Delta$ for $\textsf{ParSym}$, we have that 
\begin{align*}
 & \Delta \textsf{H}_{\begin{tikzpicture}[scale = 0.5,thick, baseline={(0,-1ex/2)}] 
 \tikzstyle{vertex} = [shape = circle, minimum size = 7pt, inner sep = 1pt] 
 \node[vertex] (G--4) at (4.5, -1) [shape = circle, draw] {}; 
 \node[vertex] (G--3) at (3.0, -1) [shape = circle, draw] {}; 
 \node[vertex] (G--2) at (1.5, -1) [shape = circle, draw] {}; 
\node[vertex] (G--1) at (0.0, -1) [shape = circle, draw] {}; 
\node[vertex] (G-1) at (0.0, 1) [shape = circle, draw] {}; 
\node[vertex] (G-2) at (1.5, 1) [shape = circle, draw] {}; 
\node[vertex] (G-3) at (3.0, 1) [shape = circle, draw] {}; 
\node[vertex] (G-4) at (4.5, 1) [shape = circle, draw] {}; 
\draw[] (G--4) .. controls +(-0.5, 0.5) and +(0.5, 0.5) .. (G--3); 
\draw[] (G--2) .. controls +(-0.5, 0.5) and +(0.5, 0.5) .. (G--1); 
\draw[] (G-1) .. controls +(0.5, -0.5) and +(-0.5, -0.5) .. (G-2); 
\draw[] (G-2) .. controls +(0.5, -0.5) and +(-0.5, -0.5) .. (G-3); 
\end{tikzpicture}} = \\
 & \textsf{H}_{\begin{tikzpicture}[scale = 0.5,thick, baseline={(0,-1ex/2)}] 
 \tikzstyle{vertex} = [shape = circle, minimum size = 7pt, inner sep = 1pt] 
 \node[vertex] (G--4) at (4.5, -1) [shape = circle, draw] {}; 
 \node[vertex] (G--3) at (3.0, -1) [shape = circle, draw] {}; 
 \node[vertex] (G--2) at (1.5, -1) [shape = circle, draw] {}; 
\node[vertex] (G--1) at (0.0, -1) [shape = circle, draw] {}; 
\node[vertex] (G-1) at (0.0, 1) [shape = circle, draw] {}; 
\node[vertex] (G-2) at (1.5, 1) [shape = circle, draw] {}; 
\node[vertex] (G-3) at (3.0, 1) [shape = circle, draw] {}; 
\node[vertex] (G-4) at (4.5, 1) [shape = circle, draw] {}; 
\draw[] (G--4) .. controls +(-0.5, 0.5) and +(0.5, 0.5) .. (G--3); 
\draw[] (G--2) .. controls +(-0.5, 0.5) and +(0.5, 0.5) .. (G--1); 
\draw[] (G-1) .. controls +(0.5, -0.5) and +(-0.5, -0.5) .. (G-2); 
\draw[] (G-2) .. controls +(0.5, -0.5) and +(-0.5, -0.5) .. (G-3); 
\end{tikzpicture}} \otimes 
 \textsf{H}_{\varnothing} + 
 \textsf{H}_{\begin{tikzpicture}[scale = 0.5,thick, baseline={(0,-1ex/2)}] 
\tikzstyle{vertex} = [shape = circle, minimum size = 7pt, inner sep = 1pt] 
\node[vertex] (G--3) at (3.0, -1) [shape = circle, draw] {}; 
\node[vertex] (G--2) at (1.5, -1) [shape = circle, draw] {}; 
\node[vertex] (G--1) at (0.0, -1) [shape = circle, draw] {}; 
\node[vertex] (G-1) at (0.0, 1) [shape = circle, draw] {}; 
\node[vertex] (G-2) at (1.5, 1) [shape = circle, draw] {}; 
\node[vertex] (G-3) at (3.0, 1) [shape = circle, draw] {}; 
\draw[] (G--2) .. controls +(-0.5, 0.5) and +(0.5, 0.5) .. (G--1); 
\draw[] (G-1) .. controls +(0.5, -0.5) and +(-0.5, -0.5) .. (G-2); 
\draw[] (G-2) .. controls +(0.5, -0.5) and +(-0.5, -0.5) .. (G-3); 
\end{tikzpicture}} \otimes \textsf{H}_{ \begin{tikzpicture}[scale = 0.5,thick, baseline={(0,-1ex/2)}] 
\tikzstyle{vertex} = [shape = circle, minimum size = 7pt, inner sep = 1pt] 
\node[vertex] (G--1) at (0.0, -1) [shape = circle, draw] {}; 
\node[vertex] (G-1) at (0.0, 1) [shape = circle, draw] {}; 
\end{tikzpicture} } + 
 \textsf{H}_{\varnothing} \otimes \textsf{H}_{\begin{tikzpicture}[scale = 0.5,thick, baseline={(0,-1ex/2)}] 
 \tikzstyle{vertex} = [shape = circle, minimum size = 7pt, inner sep = 1pt] 
 \node[vertex] (G--4) at (4.5, -1) [shape = circle, draw] {}; 
 \node[vertex] (G--3) at (3.0, -1) [shape = circle, draw] {}; 
 \node[vertex] (G--2) at (1.5, -1) [shape = circle, draw] {}; 
\node[vertex] (G--1) at (0.0, -1) [shape = circle, draw] {}; 
\node[vertex] (G-1) at (0.0, 1) [shape = circle, draw] {}; 
\node[vertex] (G-2) at (1.5, 1) [shape = circle, draw] {}; 
\node[vertex] (G-3) at (3.0, 1) [shape = circle, draw] {}; 
\node[vertex] (G-4) at (4.5, 1) [shape = circle, draw] {}; 
\draw[] (G--4) .. controls +(-0.5, 0.5) and +(0.5, 0.5) .. (G--3); 
\draw[] (G--2) .. controls +(-0.5, 0.5) and +(0.5, 0.5) .. (G--1); 
\draw[] (G-1) .. controls +(0.5, -0.5) and +(-0.5, -0.5) .. (G-2); 
\draw[] (G-2) .. controls +(0.5, -0.5) and +(-0.5, -0.5) .. (G-3); 
\end{tikzpicture}}. 
\end{align*}
\end{example}

 In our constructing a bialgebra structure for $\textsf{ParSym}$, 
 it would be appropriate to prove the coassociativity of the operation given in Definition \ref{DeltaParSym}. 

\begin{lemma}\label{bulletassociative}
 The operation $\bullet$ is associative. 
\end{lemma}

\begin{proof}
 For diagrams $\pi^{(1)}$, $\pi^{(2)}$, and $\pi^{(3)}$,  by placing $\pi^{(3)}$ to the right of $\pi^{(2)}$ and joining these diagrams with a bottom edge 
 according  to the definition of $\bullet$, and by then placing $\pi^{(1)}$ to the left of the resultant configuration and again adding a joining edge at the 
 bottom according to $\bullet$,  we obtain the product $\pi^{(1)} \bullet (\pi^{(2)} \bullet \pi^{(3)})$,  which is the same as the configuartion obtained by 
 taking $\pi^{(1)}$ and placing $\pi^{(2)}$ to the right of it and joining the diagrams according to $\bullet$, and by then placing $\pi^{(3)}$ to the right and 
 again adding a bottom edge according to $\bullet$. 
\end{proof}

\begin{theorem}
 The operation $\Delta$ is coassociative. 
\end{theorem}

\begin{proof}
 Starting with \eqref{displayDelta}, we proceed to apply $\text{id} \otimes \Delta$, writing 
\begin{equation}\label{display77D7e77lta}
\left( \text{id} \otimes \Delta \right) \left( \Delta \textsf{H}_{\pi} \right) = 
 \sum_{G_{1}} \textsf{H}_{G_{1}} \otimes \left( \sum_{G_{3}, G_{4}} \textsf{H}_{G_{3}} \otimes \textsf{H}_{G_{4}} \right),
\end{equation}
 where $G_{1}$ is as before, and where the inner sum, in accordance with 
 Definition \ref{DeltaParSym}, is over all 
 $\otimes$-irreducible partition diagrams $G_{3}$ and $G_{4}$
 such that $\textsf{H}_{G_{3}} \bullet \textsf{H}_{G_{4}} = \textsf{H}_{G_{2}}$, letting $G_{2}$ be as before. 
 So, we may rewrite \eqref{display77D7e77lta} so that 
\begin{equation}\label{dispqlqaqyq7q7Dqq7qe77lta}
 \left( \text{id} \otimes \Delta \right) \left( \Delta \textsf{H}_{\pi} \right) = 
 \sum_{G_{1}, G_{3}, G_{4}} \textsf{H}_{G_{1}} \otimes \textsf{H}_{G_{3}} \otimes \textsf{H}_{G_{4}}, 
\end{equation}
 where the sum in \eqref{dispqlqaqyq7q7Dqq7qe77lta} is over all 
 $\otimes$-irreducible partition diagrams $G_{1}$, $G_{3}$, and $G_{4}$
 such that $\textsf{H}_{G_{1}} \bullet (\textsf{H}_{G_{3}} \bullet \textsf{H}_{G_{4}} )$ 
 equals $\pi$. By Lemma \ref{bulletassociative},  the sum in \eqref{dispqlqaqyq7q7Dqq7qe77lta} 
 is the same as the corresponding sum over 
 all $\otimes$-irreducible partition diagrams $G_{1}$, $G_{3}$, and $G_{4}$
 such that $(\textsf{H}_{G_{1}} \bullet \textsf{H}_{G_{3}} ) \bullet \textsf{H}_{G_{4}}$. 
 So, by a symmetric argument, relative to our derivation of \eqref{dispqlqaqyq7q7Dqq7qe77lta}, 
 by applying $\Delta \otimes \text{id}$ to $\Delta \textsf{H}_{\pi}$, 
 we obtain the same sum in \eqref{dispqlqaqyq7q7Dqq7qe77lta}. 
\end{proof}

 By defining a counit morphism $\varepsilon\colon \textsf{ParSym} \to \mathbbm{k}$ so that $\varepsilon(\textsf{H}_{\varnothing}) = 
 1_{\mathbbm{k}}$ and $\varepsilon(\textsf{H}_{\pi})$ $=$ $ 0$ for a nonempty partition diagram $\pi$, 
   from the coassociativity of $\Delta$, we may obtain  a    bialgebra structure on $\textsf{ParSym}$.  

 The following result concerning the relationship between $\bullet$ and $\otimes$ will be useful in our 
 constructing an antipode map for $\textsf{ParSym}$. 

\begin{theorem}\label{2702730791173171P7M1A}
 Let $\rho$ and $\pi$ be two partition diagrams. Then $\rho \bullet \pi$ is $\otimes$-irreducible if and only if 
 $\rho$ and $\pi$ are $\otimes$-irreducible. 
\end{theorem}

\begin{proof}
 $(\Longrightarrow)$ Suppose that $\rho \bullet \pi$ is $\otimes$-irreducible. 
 It could then not be the case that $\rho$ has a separation, since this would result in a separation in $\rho \bullet \pi$, 
 and the same applies with respect to $\pi$. 

 \ 

\noindent $(\Longleftarrow)$ Suppose that $\rho$ and $\pi$ are $\otimes$-irreducible.   By placing $\pi$ to the right of $\rho$, and then forming an isthmus     
   joining the bottom right node of $\rho$ and the bottom left node of $\pi$,    any graph equivalent to $\rho \bullet \pi$ could not have a separating line   
  between     the nodes corresponding to $\rho$ and the nodes corresponding to $\pi$.     So, there are no separations in $\rho$ and no separations in $\pi$,    
   and there could not be a separation between   the nodes corresponding to $\rho$ and the nodes corresponding to $\pi$,    so we may conclude that $\rho   
  \bullet \pi$ is $\otimes$-irreducible.   
\end{proof}

 While the result highlighted as Theorem \ref{graphtheorem} below is to be used in a key way in our construction of a projection morphism from $ 
 \textsf{ParSym}$ to $\textsf{NSym}$, this result is of interest in its own right, since it gives us a way of formalizing how partition diagrams can be 
 ``broken up'' into components analogous to the entries of an integer composition, in a way suggested by \eqref{bulletwithcirc} and 
 \eqref{path314}. We may define $\bullet$-irreducibility by direct analogy with Definition \ref{oirreducible}. 
 Our proof of the following result is also useful in relation to the material on diagram subalgebras in Section \ref{sectionDiagramHopf} below. 

\begin{theorem}\label{graphtheorem}
 For a diagram $\pi$ that is nonempty, there is a unique decomposition 
\begin{equation}\label{piinlemma}
 \pi = \theta^{(1)} \bullet \theta^{(2)} \bullet \cdots \bullet \theta^{(m)} 
\end{equation}
 into diagrams that are $\bullet$-irreducible, and where $m = m(\pi)$ is a fixed statistic depending on $\pi$. 
\end{theorem}

\begin{proof}
 As above, we adopt the notational convention indicated in Remark \ref{realizedthiswas}. Let $i'$ and $(i+1)'$ be consecutive vertices in the bottom row of 
 $\pi$, and suppose that these bottom nodes are in the same connected component of $\pi$. For a fixed pair $(i', (i+1)')$ of bottom vertices of the 
 specified form, if it is possible to apply a construction of the following form, while maintaining 
 the connected components of $\pi$, then we let this construction be applied: If $i'$ and $(i+1)'$ are 
 non-adjacent, then we add the edge $\{ i', (i+1)' \}$, and if $i$ and $i'$ are in the same component but non-adjacent, 
 then we add the edge $\{ i, i' \}$, 
 and if $i+1$ and $(i+1)'$ are in the same component but non-adjacent, 
 then we add the edge $\{ i+1, (i+1)' \}$, 
 and, if possible, we then remove or rewrite edges so that 
 $i$ and $i+1$ are not 
 adjacent and so that no edge partly or fully appears \emph{strictly} within the area given by 
 the rectangle formed by the vertices in $\{ i, i+1, i', (i+1)' \}$, 
 and this does not include the unique edge incident with $i'$ and $(i+1)'$ and does not include the possibility of an edge incident with $i$ and $i'$ 
 and does not include the possibility of an edge incident with $i+1$ and $(i+1)'$. For the same fixed pair $(i', (i+1)')$ satisfying all of the specified conditions, 
 the edge $\{ i', (i+1)' \}$ is an isthmus, by construction. 
 By this construction, any edge $e$ removed is such that $e$ is either incident with both $i$ 
 and $i+1$ or such that $e$ partly or fully touches the strict rectangular area specified. 
 
 So, for a pair $(j', (j+1)') \neq (i', (i+1)')$ of bottom nodes satisfying the same conditions as before, 
 we let it be understood that we have already formed, as above, the isthmus $\{ i', (i+1)' \}$, in such a way so that 
 $i$ and $i+1$ are not adjacent, and in such a way so that 
 the strict rectangular region corresponding to $\{ i, i+1, i', (i+1)' \}$ is ``empty'' in the sense specified as above. 
 By then, if necessary, and according to the above procedure, 
 adding the edge $\{ j', (j+1)' \}$ and/or the edge $\{ j, j' \}$ and/or the edge $\{ j+1, (j+1)' \}$, 
 and, if necessary, removing 
 a possible edge $\{ j, j+1 \}$ or any edges partly or wholly touching the strict rectangular 
 area corresponding to $\{ j, j+1, j', (j+1)' \}$, 
 this would have no effect on the following, even for the ``borderline'' or extremal cases 
 whereby $i+1 = j$ or $j + 1 = i$: 
 (i) The edge $\{ i', (i+1)' \}$ being an isthmus; (ii) 
 The vertices $i$ and $i+1$ being non-adjacent; (iii) The vertices $i$ and $i'$ being adjacent if these vertices
 happen to be in the same component; (iv) The vertices $i+1$ and $(i+1)'$ being adjacent if these vertices happen to 
 be in the same component; and (v) The emptiness 
 of the strict rectangular region corresponding to $\{ i, i+1, i', (i+1)' \}$. 

 So, after applying our construction with respect to $(i', (i+1)')$, 
 we have shown that the application of our construction to another pair $(j', (j+1)')$ satisfying the same conditions 
 would have no effect on how our construction was initially applied to form an isthmus $\{ i', (i+1)' \}$ 
 such that the listed properties hold. 
 So, inductively, we may apply our construction over and over, with each such application 
 having no effect on the previous applications. 
 So, we repeat our edge addition/removal process wherever possible to the nonempty diagram $\pi$, and, from such successive applications, 
 we may express $\pi$ as in \eqref{piinlemma}, 
 where each application of the $\bullet$ operation in \eqref{piinlemma} has the effect 
 of adding an edge of the form $\{ k', (k+1)' \}$, for a pair $(k', (k+1)')$ of bottom nodes satisfying the above specified conditions in $\pi$, 
 and where each factor on the right of \eqref{piinlemma} is a nonempty diagram. Each factor in \eqref{piinlemma} is irreducible in the sense that it 
 cannot be the case that $\theta^{(i)} = \theta^{(j_{1})} \bullet \theta^{(j_{2})}$ for nonempty diagrams $\theta^{(j_{1})}$ and $\theta^{(j_{2})}$, 
 because, otherwise, the diagrams $\theta^{(j_{1})}$ and $ \theta^{(j_{2})}$ would have already been included in the above $\bullet$-decomposition, 
 since the edge addition/removal process described above was applied wherever possible and since 
 each application of this process is independent of any previous such applications. 
 Given a decomposition 
\begin{equation}\label{pigammagamma}
 \pi = \gamma^{(1)} \bullet \gamma^{(2)} \bullet \cdots \bullet \gamma^{(m)}, 
\end{equation}
 for $\bullet$-irreducible and nonempty diagrams of the form $\gamma^{(i)}$, 
 we can show, as follows, that the right-hand sides of 
 \eqref{piinlemma} and \eqref{pigammagamma} agree: 
 Again, the application of $\bullet$ has the effect of adding an edge $\{ i', (i+1)' \}$ 
 such that $(i', (i+1)')$ satisfies the above indicated conditions, 
 so that we may, inductively, compare 
 such edges corresponding to \eqref{piinlemma} and \eqref{pigammagamma}. 
 So, there is a unique $\bullet$-decomposition of $\pi$ 
 of the desired form. 
\end{proof}

 Apart from our applying Theorem \ref{graphtheorem} 
 to prove Theorem \ref{theoremproj}, Theorem \ref{graphtheorem} is also of interest in 
 terms of how it can be used to formalize and give light to how the $\bullet$ operation in Definition \ref{bulletdefinition} is appropriate and natural as a lifting 
 of near-concatenation. For example, one might think that there could be 
 advantages of using a variant of $\bullet$ given by, say, joining top and bottom edges, letting $\odot$ be such that $\pi \odot \rho$ is the graph 
 obtained by adding an edge incident with the top left node of $\pi$ and the top right node of $\rho$, and by adding an edge incident with the bottom 
 left node of $\pi$ and the bottom right node of $\rho$. However, this operation $\odot$ would not provide the uniqueness property indicated in 
 Theorem \ref{graphtheorem}, as shown in Example \ref{frompen} below. 

\begin{example}\label{frompen}
 The uniqueness property in Theorem \ref{graphtheorem} may be illustrated by evaluating the $\bullet$-products 
 $$ 
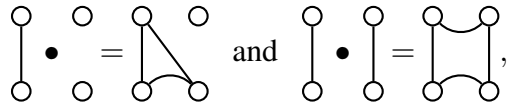
\begin{tikzpicture}[scale = 0.5,thick, baseline={(0,-1ex/2)}] 
\tikzstyle{vertex} = [shape = circle, minimum size = 7pt, inner sep = 1pt] 
\node[vertex] (G--1) at (0.0, -1) [shape = circle, draw] {}; 
\node[vertex] (G-1) at (0.0, 1) [shape = circle, draw] {}; 
\draw[] (G-1) .. controls +(0, -1) and +(0, 1) .. (G--1); 
\end{tikzpicture} \ \bullet \ \begin{tikzpicture}[scale = 0.5,thick, baseline={(0,-1ex/2)}] 
\tikzstyle{vertex} = [shape = circle, minimum size = 7pt, inner sep = 1pt] 
\node[vertex] (G--1) at (0.0, -1) [shape = circle, draw] {}; 
\node[vertex] (G-1) at (0.0, 1) [shape = circle, draw] {}; 
\end{tikzpicture}
 = \begin{tikzpicture}[scale = 0.5,thick, baseline={(0,-1ex/2)}] 
\tikzstyle{vertex} = [shape = circle, minimum size = 7pt, inner sep = 1pt] 
\node[vertex] (G--2) at (1.5, -1) [shape = circle, draw] {}; 
\node[vertex] (G--1) at (0.0, -1) [shape = circle, draw] {}; 
\node[vertex] (G-1) at (0.0, 1) [shape = circle, draw] {}; 
\node[vertex] (G-2) at (1.5, 1) [shape = circle, draw] {}; 
\draw[] (G-1) .. controls +(0.75, -1) and +(-0.75, 1) .. (G--2); 
\draw[] (G--2) .. controls +(-0.5, 0.5) and +(0.5, 0.5) .. (G--1); 
\draw[] (G--1) .. controls +(0, 1) and +(0, -1) .. (G-1); 
\end{tikzpicture} \ \ \ \text{and} \ \ \ 
 \begin{tikzpicture}[scale = 0.5,thick, baseline={(0,-1ex/2)}] 
\tikzstyle{vertex} = [shape = circle, minimum size = 7pt, inner sep = 1pt] 
\node[vertex] (G--1) at (0.0, -1) [shape = circle, draw] {}; 
\node[vertex] (G-1) at (0.0, 1) [shape = circle, draw] {}; 
\draw[] (G-1) .. controls +(0, -1) and +(0, 1) .. (G--1); 
\end{tikzpicture} \ \bullet \ \begin{tikzpicture}[scale = 0.5,thick, baseline={(0,-1ex/2)}] 
\tikzstyle{vertex} = [shape = circle, minimum size = 7pt, inner sep = 1pt] 
\node[vertex] (G--1) at (0.0, -1) [shape = circle, draw] {}; 
\node[vertex] (G-1) at (0.0, 1) [shape = circle, draw] {}; 
\draw[] (G-1) .. controls +(0, -1) and +(0, 1) .. (G--1); 
\end{tikzpicture} = \begin{tikzpicture}[scale = 0.5,thick, baseline={(0,-1ex/2)}] 
\tikzstyle{vertex} = [shape = circle, minimum size = 7pt, inner sep = 1pt] 
\node[vertex] (G--2) at (1.5, -1) [shape = circle, draw] {}; 
\node[vertex] (G--1) at (0.0, -1) [shape = circle, draw] {}; 
\node[vertex] (G-1) at (0.0, 1) [shape = circle, draw] {}; 
\node[vertex] (G-2) at (1.5, 1) [shape = circle, draw] {}; 
\draw[] (G-1) .. controls +(0.5, -0.5) and +(-0.5, -0.5) .. (G-2); 
\draw[] (G-2) .. controls +(0, -1) and +(0, 1) .. (G--2); 
\draw[] (G--2) .. controls +(-0.5, 0.5) and +(0.5, 0.5) .. (G--1); 
\draw[] (G--1) .. controls +(0, 1) and +(0, -1) .. (G-1); 
\end{tikzpicture}, $$
 in contrast to the non-uniqueness suggested by the $\odot$-products shown below: 
 $$ \begin{tikzpicture}[scale = 0.5,thick, baseline={(0,-1ex/2)}] 
\tikzstyle{vertex} = [shape = circle, minimum size = 7pt, inner sep = 1pt] 
\node[vertex] (G--1) at (0.0, -1) [shape = circle, draw] {}; 
\node[vertex] (G-1) at (0.0, 1) [shape = circle, draw] {}; 
\draw[] (G-1) .. controls +(0, -1) and +(0, 1) .. (G--1); 
\end{tikzpicture} \ \odot \ \begin{tikzpicture}[scale = 0.5,thick, baseline={(0,-1ex/2)}] 
\tikzstyle{vertex} = [shape = circle, minimum size = 7pt, inner sep = 1pt] 
\node[vertex] (G--1) at (0.0, -1) [shape = circle, draw] {}; 
\node[vertex] (G-1) at (0.0, 1) [shape = circle, draw] {}; 
\end{tikzpicture}
 = \begin{tikzpicture}[scale = 0.5,thick, baseline={(0,-1ex/2)}] 
\tikzstyle{vertex} = [shape = circle, minimum size = 7pt, inner sep = 1pt] 
\node[vertex] (G--2) at (1.5, -1) [shape = circle, draw] {}; 
\node[vertex] (G--1) at (0.0, -1) [shape = circle, draw] {}; 
\node[vertex] (G-1) at (0.0, 1) [shape = circle, draw] {}; 
\node[vertex] (G-2) at (1.5, 1) [shape = circle, draw] {}; 
\draw[] (G-1) .. controls +(0.5, -0.5) and +(-0.5, -0.5) .. (G-2); 
\draw[] (G-2) .. controls +(0, -1) and +(0, 1) .. (G--2); 
\draw[] (G--2) .. controls +(-0.5, 0.5) and +(0.5, 0.5) .. (G--1); 
\draw[] (G--1) .. controls +(0, 1) and +(0, -1) .. (G-1); 
\end{tikzpicture} \ \ \ \text{and} \ \ \ 
 \begin{tikzpicture}[scale = 0.5,thick, baseline={(0,-1ex/2)}] 
\tikzstyle{vertex} = [shape = circle, minimum size = 7pt, inner sep = 1pt] 
\node[vertex] (G--1) at (0.0, -1) [shape = circle, draw] {}; 
\node[vertex] (G-1) at (0.0, 1) [shape = circle, draw] {}; 
\draw[] (G-1) .. controls +(0, -1) and +(0, 1) .. (G--1); 
\end{tikzpicture} \ \odot \ \begin{tikzpicture}[scale = 0.5,thick, baseline={(0,-1ex/2)}] 
\tikzstyle{vertex} = [shape = circle, minimum size = 7pt, inner sep = 1pt] 
\node[vertex] (G--1) at (0.0, -1) [shape = circle, draw] {}; 
\node[vertex] (G-1) at (0.0, 1) [shape = circle, draw] {}; 
\draw[] (G-1) .. controls +(0, -1) and +(0, 1) .. (G--1); 
\end{tikzpicture} = \begin{tikzpicture}[scale = 0.5,thick, baseline={(0,-1ex/2)}] 
\tikzstyle{vertex} = [shape = circle, minimum size = 7pt, inner sep = 1pt] 
\node[vertex] (G--2) at (1.5, -1) [shape = circle, draw] {}; 
\node[vertex] (G--1) at (0.0, -1) [shape = circle, draw] {}; 
\node[vertex] (G-1) at (0.0, 1) [shape = circle, draw] {}; 
\node[vertex] (G-2) at (1.5, 1) [shape = circle, draw] {}; 
\draw[] (G-1) .. controls +(0.5, -0.5) and +(-0.5, -0.5) .. (G-2); 
\draw[] (G-2) .. controls +(0, -1) and +(0, 1) .. (G--2); 
\draw[] (G--2) .. controls +(-0.5, 0.5) and +(0.5, 0.5) .. (G--1); 
\draw[] (G--1) .. controls +(0, 1) and +(0, -1) .. (G-1); 
\end{tikzpicture}. $$
\end{example}

 Theorem \ref{graphtheorem} is also closely related to our construction of an antipode for $\textsf{ParSym}$ and our construction of an analogue of the 
 $E$-basis of $\textsf{NSym}$, as in Section \ref{subsectionES}. 
 
\subsection{An elementary-like basis and an antipode}\label{subsectionES}
 The convolution identity such that 
\begin{equation}\label{convolutionSym}
 \sum_{i=0}^{n} (-1)^{i} e_{i} h_{n-i} = \delta_{n,0} 
\end{equation}
 is of basic importance in the theory of symmetric functions, letting commutative $h$-generators be as in \eqref{displaySym}, and where elementary 
 generators for $\textsf{Sym}$ bay be recursively defined according to \eqref{convolutionSym}. The elementary basis $\{ E_{\alpha} : \alpha \in 
 \mathcal{C} \}$ of $\textsf{NSym}$ may be defined by direct analogy with \eqref{convolutionSym}, by setting 
\begin{equation}\label{Erecursion}
 E_{n} = \sum_{i = 1}^{n} (-1)^{i+1} H_{i} E_{n - i}, 
\end{equation}
 with $E_{\alpha} = E_{\alpha_{1}} E_{\alpha_{2}} \cdots E_{\alpha_{\ell(\alpha)}}$. Successive applications of 
 \eqref{Erecursion} lead to 
\begin{equation}\label{EtoHNSym}
 E_{n} = \sum_{\beta \vDash n} (-1)^{\ell(\beta) +n} H_{\beta}, 
\end{equation}
 and we intend to determine an analogue, for $\textsf{ParSym}$, of \eqref{EtoHNSym}, 
 to define and evaluate an antipode antihomomorphism for $\textsf{ParSym}$. 

 We may define the Hopf algebra morphism $S_{\textsf{NSym}}\colon \textsf{NSym} \to \textsf{NSym}$ so that 
\begin{equation}\label{SNSymH}
 S_{\textsf{NSym}}(H_{n}) = (-1)^{n} E_{n}.
\end{equation}
 So, by applying the morphism $\text{id} \otimes S_{\textsf{NSym}}$ 
 to both sides of \eqref{DeltaNSym}, 
 and then applying the multiplicative operation of $\textsf{NSym}$, 
 as a morphism from $\textsf{NSym} \otimes \textsf{NSym} \to \textsf{NSym}$, we obtain, 
 in view of \eqref{Erecursion}, that 
\begin{align*}
 \sum_{i = 0}^{n} (-1)^{n-i} H_{i} E_{n-i} 
 & = \begin{cases} 
 1_{\textsf{NSym}} & \text{if $n = 0$,} \\ 
 0 & \text{otherwise,} 
 \end{cases} \\ 
 & = \eta_{\textsf{NSym}} \left( \varepsilon_{\textsf{NSym}}\left( H_{n} \right) \right), 
\end{align*}
 and a symmetric argument may be used to complete a proof that the following diagram commutes, 
 for $\mathcal{H} = \textsf{NSym}$. 
\begin{equation}\label{categoryHopf}
\begin{tikzcd}[row sep=3.6em,column sep=1em]
& \mathcal{H}\otimes \mathcal{H} \arrow[rr,"S\otimes\id"] && \mathcal{H}\otimes \mathcal{H} \arrow[dr,"\nabla"] \\
\mathcal{H} \arrow[ur,"\Delta"] \arrow[rr,"\varepsilon"] \arrow[dr,"\Delta"'] && \mathbbm{k} \arrow[rr,"\eta"] && \mathcal{H} \\
& \mathcal{H}\otimes \mathcal{H} \arrow[rr,"\id\otimes S"'] && \mathcal{H}\otimes \mathcal{H} \arrow[ur,"\nabla"']
\end{tikzcd}
\end{equation}

 The antipode antihomomorphism for $\textsf{NSym}$ may be defined in an equivalent way so that 
\begin{equation}\label{SHn}
 S(H_{n}) = \sum_{\alpha} (-1)^{\ell(\alpha)} H_{\alpha_{1}} H_{\alpha_{2}} \cdots H_{\alpha_{\ell(\alpha)}}, 
\end{equation}
 where the sum is over all tuples $\alpha$ of positive integers such that 
 $H_{\alpha_{1}} \bullet H_{\alpha_{2}} \bullet \cdots \bullet H_{\alpha_{\ell(\alpha)}} = H_{(n)}$. 
 Writing $S_{\textsf{ParSym}} = S$, we set $S(\textsf{H}_{\varnothing}) = \textsf{H}_{\varnothing}$, 
 and, by direct analogy with \eqref{SHn}, 
 for a nonempty, irreducible partition diagram $\pi$, we define 
\begin{equation}\label{antipodeirreducible}
 S\left( \textsf{H}_{\pi} \right) = \sum (-1)^{\ell} \textsf{H}_{\rho^{(1)}} 
 \textsf{H}_{\rho^{(2)}} \cdots \textsf{H}_{\rho^{(\ell)}}, 
\end{equation}
 where the sum in \eqref{antipodeirreducible} is over all possible tuples $( \rho^{(1)}, \rho^{(2)}, \ldots, \rho^{(\ell)} )$ of partition 
 diagrams such that $\textsf{H}_{\rho^{(1)}} \bullet \textsf{H}_{\rho^{(2)}} \bullet \cdots \bullet \textsf{H}_{\rho^{(\ell)}} = \textsf{H}_{\pi}$. 
 By Theorem \ref{2702730791173171P7M1A}, the $\otimes$-irreducibility of 
 each $\rho$-factor follows from the $\otimes$-irreducibility of $\pi$. 
 We extend the mapping in \eqref{antipodeirreducible} linearly
 and so as to obtain an \emph{antimorphism} of algebras, with $S( \textsf{H}_{\pi} \textsf{H}_{\rho} ) 
 = S(\textsf{H}_{\rho}) S(\textsf{H}_{\pi})$ for irreducible diagrams $\pi$ and $\rho$. 

\begin{example}
 In view of coproduct expansion in Example \ref{exampleDelta} and the commutative diagram shown in \eqref{categoryHopf}, we consider the 
 application of $S \otimes \text{id}$ to the right-hand side of the equality in Example \ref{exampleDelta}. 
\begin{align*}
 & \Bigg( \ \raisebox{1.4ex}{$ -\textsf{H}_{\begin{tikzpicture}[scale = 0.5,thick, baseline={(0,-1ex/2)}] 
 \tikzstyle{vertex} = [shape = circle, minimum size = 7pt, inner sep = 1pt] 
 \node[vertex] (G--4) at (4.5, -1) [shape = circle, draw] {}; 
 \node[vertex] (G--3) at (3.0, -1) [shape = circle, draw] {}; 
 \node[vertex] (G--2) at (1.5, -1) [shape = circle, draw] {}; 
\node[vertex] (G--1) at (0.0, -1) [shape = circle, draw] {}; 
\node[vertex] (G-1) at (0.0, 1) [shape = circle, draw] {}; 
\node[vertex] (G-2) at (1.5, 1) [shape = circle, draw] {}; 
\node[vertex] (G-3) at (3.0, 1) [shape = circle, draw] {}; 
\node[vertex] (G-4) at (4.5, 1) [shape = circle, draw] {}; 
\draw[] (G--4) .. controls +(-0.5, 0.5) and +(0.5, 0.5) .. (G--3); 
\draw[] (G--2) .. controls +(-0.5, 0.5) and +(0.5, 0.5) .. (G--1); 
\draw[] (G-1) .. controls +(0.5, -0.5) and +(-0.5, -0.5) .. (G-2); 
\draw[] (G-2) .. controls +(0.5, -0.5) and +(-0.5, -0.5) .. (G-3); 
\end{tikzpicture}} + \textsf{H}_{\begin{tikzpicture}[scale = 0.5,thick, baseline={(0,-1ex/2)}] 
\tikzstyle{vertex} = [shape = circle, minimum size = 7pt, inner sep = 1pt] 
\node[vertex] (G--4) at (4.5, -1) [shape = circle, draw] {}; 
\node[vertex] (G--3) at (3.0, -1) [shape = circle, draw] {}; 
\node[vertex] (G--2) at (1.5, -1) [shape = circle, draw] {}; 
\node[vertex] (G--1) at (0.0, -1) [shape = circle, draw] {}; 
\node[vertex] (G-1) at (0.0, 1) [shape = circle, draw] {}; 
\node[vertex] (G-2) at (1.5, 1) [shape = circle, draw] {}; 
\node[vertex] (G-3) at (3.0, 1) [shape = circle, draw] {}; 
\node[vertex] (G-4) at (4.5, 1) [shape = circle, draw] {}; 
\draw[] (G--2) .. controls +(-0.5, 0.5) and +(0.5, 0.5) .. (G--1); 
\draw[] (G-1) .. controls +(0.5, -0.5) and +(-0.5, -0.5) .. (G-2); 
\draw[] (G-2) .. controls +(0.5, -0.5) and +(-0.5, -0.5) .. (G-3); 
\end{tikzpicture}} $} \ \Bigg) \otimes \textsf{H}_{\varnothing} \\
 & - \textsf{H}_{\begin{tikzpicture}[scale = 0.5,thick, baseline={(0,-1ex/2)}] 
\tikzstyle{vertex} = [shape = circle, minimum size = 7pt, inner sep = 1pt] 
\node[vertex] (G--3) at (3.0, -1) [shape = circle, draw] {}; 
\node[vertex] (G--2) at (1.5, -1) [shape = circle, draw] {}; 
\node[vertex] (G--1) at (0.0, -1) [shape = circle, draw] {}; 
\node[vertex] (G-1) at (0.0, 1) [shape = circle, draw] {}; 
\node[vertex] (G-2) at (1.5, 1) [shape = circle, draw] {}; 
\node[vertex] (G-3) at (3.0, 1) [shape = circle, draw] {}; 
\draw[] (G--2) .. controls +(-0.5, 0.5) and +(0.5, 0.5) .. (G--1); 
\draw[] (G-1) .. controls +(0.5, -0.5) and +(-0.5, -0.5) .. (G-2); 
\draw[] (G-2) .. controls +(0.5, -0.5) and +(-0.5, -0.5) .. (G-3); 
\end{tikzpicture}} \otimes \textsf{H}_{ \begin{tikzpicture}[scale = 0.5,thick, baseline={(0,-1ex/2)}] 
\tikzstyle{vertex} = [shape = circle, minimum size = 7pt, inner sep = 1pt] 
\node[vertex] (G--1) at (0.0, -1) [shape = circle, draw] {}; 
\node[vertex] (G-1) at (0.0, 1) [shape = circle, draw] {}; 
\end{tikzpicture} } + \textsf{H}_{\varnothing} \otimes \textsf{H}_{\begin{tikzpicture}[scale = 0.5,thick, baseline={(0,-1ex/2)}] 
 \tikzstyle{vertex} = [shape = circle, minimum size = 7pt, inner sep = 1pt] 
 \node[vertex] (G--4) at (4.5, -1) [shape = circle, draw] {}; 
 \node[vertex] (G--3) at (3.0, -1) [shape = circle, draw] {}; 
 \node[vertex] (G--2) at (1.5, -1) [shape = circle, draw] {}; 
\node[vertex] (G--1) at (0.0, -1) [shape = circle, draw] {}; 
\node[vertex] (G-1) at (0.0, 1) [shape = circle, draw] {}; 
\node[vertex] (G-2) at (1.5, 1) [shape = circle, draw] {}; 
\node[vertex] (G-3) at (3.0, 1) [shape = circle, draw] {}; 
\node[vertex] (G-4) at (4.5, 1) [shape = circle, draw] {}; 
\draw[] (G--4) .. controls +(-0.5, 0.5) and +(0.5, 0.5) .. (G--3); 
\draw[] (G--2) .. controls +(-0.5, 0.5) and +(0.5, 0.5) .. (G--1); 
\draw[] (G-1) .. controls +(0.5, -0.5) and +(-0.5, -0.5) .. (G-2); 
\draw[] (G-2) .. controls +(0.5, -0.5) and +(-0.5, -0.5) .. (G-3); 
\end{tikzpicture}}. 
\end{align*}
\end{example}

 To evaluate the algebra antimorphism $S$, 
 one way of going about with this would be through the use of \emph{Takeuchi's formula} 
 \cite{Takeuchi1971}, by analogy with bijective methods employed by Benedetti and Sagan 
 \cite{BenedettiSagan2017} through the use of Takeuchi's formula. 
 Following \cite{BenedettiSagan2017} (cf.\ \cite{Takeuchi1971}), if $H$ is a bialgebra that is connected and graded, 
 then $H$ is a Hopf algebra, with an antipode given as follows. 
 Let the projection map $\pi\colon H \to H$ be extended linearly so that 
 $\pi$ restricted to the graded component $H_{n}$ is the zero map 
 if $n = 0$ and is the identity map for $n \geq 1$. With this setup, the antipode for $S$ is such that 
\begin{equation}\label{2777877707273707971717471797PM1A}
 S = \sum_{k \geq 0} (-1)^{k} \nabla^{k-1} \pi^{\otimes k} \Delta^{k-1}, 
\end{equation}
 with the understanding that $\nabla^{-1} = u$ and $\Delta^{-1} = \varepsilon$, 
 and where $\nabla$ denotes the multiplicative operation for $H$. 
 The formulation in \eqref{2777877707273707971717471797PM1A} of 
 Takeuchi's formula, as applied to the graded, connected bialgebra $\textsf{ParSym}$ 
 and to $\textsf{H}_{\pi}$ for an irreducible 
 partition diagram $\pi$, 
 can be used to obtain \eqref{antipodeirreducible}, 
 giving us that \eqref{antipodeirreducible}, extended linearly and as an antimorphism, 
 does indeed give us an antipode such that the diagram in 
 \eqref{categoryHopf} commutes. An alternative, bijective approach 
 toward proving this result is given below. 

\begin{theorem}\label{theoremcategory}
 The diagram in \eqref{categoryHopf} commutes, for the specified antimorphism $S$. 
\end{theorem}

\begin{proof}
 Let $\pi$ denote a nonempty, $\otimes$-irreducible 
 diagram. From \eqref{displayDelta}, by applying $S \otimes \text{id}$ to both sides of this equality, we obtain 
\begin{align}
 \left( S \otimes \text{id} \right) \left( \Delta \textsf{H}_{\pi} \right) 
 & = \sum_{G_{1}, G_{2}} S\left( \textsf{H}_{G_{1}} \right) \otimes \textsf{H}_{G_{2}} \nonumber \\ 
 & = \sum_{G_{1}, G_{2}} 
 \left( \sum (-1)^{\ell} \textsf{H}_{\rho^{(1)}} 
 \textsf{H}_{\rho^{(2)}} \cdots \textsf{H}_{\rho^{(\ell)}} \right) 
 \otimes \textsf{H}_{G_{2}}, \label{firstlabelcategory}
\end{align}
 where the inner sum in \eqref{firstlabelcategory} is such that: If $G_{1}$ is empty, then $S\left( \textsf{H}_{\varnothing} \right) = 
 \textsf{H}_{\varnothing}$, so that the inner sum, in this case, is indexed by the empty set, and if $G_{1}$ is nonempty, then the inner sum is 
 is over all possible tuples $( \rho^{(1)}, \rho^{(2)}, \ldots, \rho^{(\ell)} )$ of nonempty partition diagrams 
 that are irreducible with respect to $\otimes$, by Theorem \ref{2702730791173171P7M1A}, and 
 that are such that 
 $\textsf{H}_{\rho^{(1)}} \bullet \textsf{H}_{\rho^{(2)}} \bullet \cdots \bullet \textsf{H}_{\rho^{(\ell)}} = \textsf{H}_{G_{1}}$, again for 
 $\otimes$-irreducible partition diagrams $G_{1}$ and $G_{2}$ such that $\textsf{H}_{G_{1}} \bullet \textsf{H}_{G_{2}} = \textsf{H}_{\pi}$. 
 Applying the multiplicative 
 operation $\nabla = \nabla_{\textsf{ParSym}}$ from $\textsf{ParSym} \otimes \textsf{ParSym}$ to $\textsf{ParSym}$, we may write 
\begin{equation}\label{givesriseinvolution}
 \nabla \left( \left( S \otimes \text{id} \right) \left( \Delta \textsf{H}_{\pi} \right) \right) = \sum_{G_{1}, G_{2}} 
 \sum (-1)^{\ell} \textsf{H}_{\rho^{(1)}} \textsf{H}_{\rho^{(2)}} \cdots \textsf{H}_{\rho^{(\ell)}} \textsf{H}_{G_{2}}, 
\end{equation}
 where the inner sum in \eqref{givesriseinvolution} is as before. This gives rise to a sign-reversing involution, as indicated below. 

 According to the index set for the double sum in \eqref{givesriseinvolution}, we fix $G_{2}$, and, for this fixed graph, we then associate 
 to it a tuple $( \rho^{(1)}$, $ \rho^{(2)}$, $ \ldots$, $ \rho^{(\ell)} )$ of nonempty partition 
 diagrams that are irreducible with respect to $\otimes$, by Theorem \ref{2702730791173171P7M1A}, 
 and that are such that $ \textsf{H}_{\rho^{(1)}} \bullet \textsf{H}_{\rho^{(2)}} \bullet \cdots \bullet \textsf{H}_{\rho^{(\ell)}} \bullet 
 \textsf{H}_{G_{2}} = \pi$, so that 
\begin{equation}\label{agiventerm}
 (-1)^{\ell} \textsf{H}_{\rho^{(1)}} \textsf{H}_{\rho^{(2)}} \cdots \textsf{H}_{\rho^{(\ell)}} 
 \textsf{H}_{G_{2}} 
\end{equation}
 is a given term appearing in \eqref{givesriseinvolution}. 
 If $G_{2}$ is nonempty, 
 then we let the term indicated in \eqref{agiventerm} be mapped to 
\begin{equation}\label{secondgiventerm}
 (-1)^{\ell+1} \left( \textsf{H}_{\rho^{(1)}} \textsf{H}_{\rho^{(2)}} \cdots \textsf{H}_{\rho^{(\ell)}} 
 \textsf{H}_{\rho^{(\ell + 1)}} \right) \textsf{H}_{G_{3}}, 
\end{equation}
 where $\rho^{(\ell + 1)} = G_{2}$ and where $G_{3} = \varnothing$, 
 giving a term appearing in \eqref{givesriseinvolution}. 
 Conversely, for terms of \eqref{givesriseinvolution} such that $G_{2} = \varnothing$, 
 we let such terms 
 be mapped in such a way so that a term of the form indicated in \eqref{secondgiventerm} gets mapped back to \eqref{agiventerm}. This gives us a 
 sign-reversing involution that gives us that \eqref{givesriseinvolution} vanishes if $\pi$ is nonempty. If $\pi$ is empty, then 
 $\Delta \textsf{H}_{\varnothing} = \textsf{H}_{\varnothing} \otimes \textsf{H}_{\varnothing}$, 
 so that the application of $S \otimes \text{id}$ again gives $\textsf{H}_{\varnothing} \otimes \textsf{H}_{\varnothing}$, 
 so that the application of $\nabla$ then gives $1_{\textsf{ParSym}}$. 
 So, we have shown that 
\begin{align*}
 \nabla \left( \left( S \otimes \text{id} \right) \left( \Delta \textsf{H}_{\pi} \right) \right) 
 & = \begin{cases} 
 1_{\textsf{ParSym}} & \text{if $\pi = \varnothing$}, \\ 
 0 & \text{if $\pi \neq \varnothing$},
 \end{cases} \\ 
 & = \eta \left( \varepsilon\left( H_{\pi} \right) \right), 
\end{align*}
 as desired. A symmetric argument gives us the same evaluation for 
$ \nabla ( ( \text{id} \otimes S ) ( \Delta \textsf{H}_{\pi}))$. 

 So, we have shown that the desired diagram commutes, with respect to generators $\textsf{H}_{\pi}$, 
 for an $\otimes$-irreducible diagram $\pi$. Using the property that $S$ is defined as an antimorphism, 
 together with the morphisms $\Delta$, $\nabla$, etc., 
 we may obtain that the aforementioned diagram commutes in full generality. 
\end{proof}

 Theorem \ref{theoremcategory} gives us that $\textsf{ParSym}$, with its specified morphisms, is a Hopf algebra. If we consider the antipode relation for 
 $\textsf{NSym}$ shown in \eqref{SNSymH} and compare it to the definition of $S = S_{\textsf{ParSym}}$ in \eqref{antipodeirreducible}, this 
 motivates the construction of an elementary-like basis for $\textsf{ParSym}$, as below. 

\begin{definition}\label{textsfE}
 Let $\pi$ be an irreducible partition diagram of order $n$. We define 
\begin{equation}\label{sfEtosfH}
 \textsf{E}_{\pi} = \sum (-1)^{\ell + n} \textsf{H}_{\rho^{(1)}} 
 \textsf{H}_{\rho^{(2)}} \cdots \textsf{H}_{\rho^{(\ell)}}, 
\end{equation}
 where the sum in \eqref{sfEtosfH} is over the same index set in \eqref{antipodeirreducible}. For irreducible partition diagrams $\pi^{(1)}$, $\pi^{(2)}$, $ 
 \ldots$, $\pi^{(p)}$, we then set $\textsf{E}_{ \pi^{(1)} \otimes \pi^{(2)} \otimes \cdots \otimes \pi^{(p)} } = \textsf{E}_{\pi^{(1)}} 
 \textsf{E}_{\pi^{(2)}} \cdots \textsf{E}_{\pi^{(p)}}$. 
\end{definition}
 
\begin{theorem}
 The family $\{ \textsf{\emph{E}}_{\pi} : \text{$\pi$ is a partition diagram of order $n$} \}$
 is a basis of $\textsf{\emph{ParSym}}_{n}$. 
\end{theorem}

\begin{proof}
 The antipode of a graded, connected Hopf algebra is always bijective. So, the set of expressions of the form
 $S(\textsf{H}_{\pi})$, for irreducible partition diagrams $\pi$, freely generate $\textsf{ParSym}$, 
 which is equivalent to the desired result. 
\end{proof}
 
\subsection{The character of $\textsf{ParSym}$}
 A \emph{Combinatorial Hopf Algebra}, according to Aguiar, Bergeron, and Sottile \cite{AguiarBergeronSottile2006}, is a graded, connected Hopf $ 
 \mathbbm{k}$-algebra $\mathcal{H}$ together with a multiplicative, linear map of the form $\zeta\colon \mathcal{H} \to \mathbbm{k}$ referred to as the 
 \emph{character} of $\mathcal{H}$. The direct sum decomposition of $\textsf{ParSym}$ indicated in \eqref{ParSymdefinition} is such that: 

\begin{enumerate}

\item For homogeneous generators $\textsf{H}_{\pi} \in \textsf{ParSym}_{i}$ and $\textsf{H}_{\rho} \in \textsf{ParSym}_{j}$, the product 
 $\textsf{H}_{\pi} \textsf{H}_{\rho}$ is in $\textsf{ParSym}_{i+j}$, according to the horizontal concatenation operation indicated in 
 \eqref{productParSym} together with $\textsf{ParSym}_{i+j}$ being spanned by all $\textsf{H}$-basis elements indexed by partition diagrams 
 of order $i +j$; 

\item For a homogeneous generator $\textsf{H}_{\pi} \in \textsf{ParSym}_{n}$, 
 the coproduct $\Delta\left( \textsf{H}_{\pi} \right)$ is in $\bigoplus_{i + j = n} \textsf{ParSym}_{i}
 \otimes \textsf{ParSym}_{j}$, according to Definition \ref{DeltaParSym}. 

\end{enumerate}

 We thus obtain a graded Hopf algebra structure on $\textsf{ParSym}$. Since $\textsf{ParSym}_{0} = \text{span}_{\mathbbm{k}}\{ 
 \textsf{H}_{\varnothing} \}$, and since $\textsf{H}_{\varnothing}$ may be identified with $1_{\mathbbm{k}}$ in the sense indicated in 
 \eqref{etaParSym}, the graded component $\textsf{ParSym}_{0}$ may be identified with $\mathbbm{k}$, giving us that $\textsf{ParSym}$ is connected 
 as a graded Hopf algebra. 

 Following the seminal reference on CHAs \cite{AguiarBergeronSottile2006}, along with further references as in \cite{HumpertMartin2012} 
 and \cite[\S2.2]{Li2018}, a multiplicative linear map $\zeta$ from a polynomial $\mathbbm{k}$-algebra 
 to $\mathbbm{k}$ or from a $\mathbbm{k}$-algebra on power series to $\mathbbm{k}$ is, canonically, 
 given by setting one variable to $1$ and the remaining variables to $0$. 
 In particular, for the canonical character $\zeta\colon \mathbbm{k}[x_{1}, x_{2}, \ldots] \to \mathbbm{k}$, 
 this is given by setting $\zeta(x_{1}) = 1$ and $\zeta(x_{i}) = 0$ for $i \neq 1$
 and by extending these relations linearly and multiplicatively. 
 So, we may set 
\begin{equation}\label{zetaNSym}
 \zeta_{\textsf{NSym}}(H_{\alpha}) 
 = \begin{cases} 
 1 & \text{if $\alpha = (1)$ or $\alpha = ()$}, \\ 
 0 & \text{otherwise}, 
 \end{cases}
\end{equation}
 for a composition $\alpha$. By direct analogy with \eqref{zetaNSym}, we set 
\begin{equation}\label{zetaParSym}
 \zeta_{\textsf{ParSym}}(\textsf{H}_{\pi}) 
 = \begin{cases} 
 1 & \text{if $\pi = \begin{tikzpicture}[scale = 0.5,thick, baseline={(0,-1ex/2)}] 
\tikzstyle{vertex} = [shape = circle, minimum size = 7pt, inner sep = 1pt] 
\node[vertex] (G--1) at (0.0, -1) [shape = circle, draw] {}; 
\node[vertex] (G-1) at (0.0, 1) [shape = circle, draw] {}; 
\end{tikzpicture}$ or $\pi = \varnothing$}, \\ 
 0 & \text{otherwise}, 
 \end{cases}
\end{equation}
 giving us a CHA structure for $\textsf{ParSym}$. 

\subsection{Relationship with $\textsf{NSym}$}
 A \emph{combinatorial Hopf morphism} \cite[\S2.2]{Li2018} from a CHA $(\mathcal{H}_{1}, \zeta_{1})$ to a CHA $(\mathcal{H}_{2}, \zeta_{2})$ is a 
 Hopf algebra morphism $\Phi\colon \mathcal{H}_{1} \to \mathcal{H}_{2}$ such that $\zeta_{1} = \zeta_{2} \circ \Phi$. A Hopf algebra homomorphism 
 $\Phi\colon \mathcal{H}_{1} \to \mathcal{H}_{2}$ is such that $\Phi$ is an algebra homomorphism such that 
\begin{equation}\label{PhitensorPhi} 
 \Delta_{2} \circ \Phi = (\Phi \otimes \Phi) \circ \Delta_{1} 
\end{equation}
 and 
\begin{equation}\label{epsilonepsilon}
 \varepsilon_{2} \circ \Phi = \varepsilon_{1}, 
\end{equation}
 with \eqref{PhitensorPhi} and \eqref{epsilonepsilon} yielding a coalgebra morphism, where $\Delta_{1}$ and $\Delta_{2}$ respectively denote the 
 coproduct operations for $\mathcal{H}_{1}$ and $\mathcal{H}_{2}$, and where $\varepsilon_{1}$ and $\varepsilon_{2}$ denote the counit 
 morphisms of $\mathcal{H}_{1}$ and $\mathcal{H}_{2}$. 

 Define 
\begin{equation}\label{Phidefinition}
 \Phi\colon \textsf{NSym} \to \textsf{ParSym} 
\end{equation}
 so that 
\begin{equation}\label{PhiHn}
 \Phi(H_{n}) = \textsf{H}_{ \begin{tikzpicture}[scale = 0.5,thick, baseline={(0,-1ex/2)}] 
\tikzstyle{vertex} = [shape = circle, minimum size = 7pt, inner sep = 1pt] 
\node[vertex] (G--3) at (3.0, -1) [shape = circle, draw] {}; 
\node[vertex] (G--2) at (1.5, -1) [shape = circle, draw] {}; 
\node[vertex] (G--1) at (0.0, -1) [shape = circle, draw] {}; 
\node[vertex] (G-1) at (0.0, 1) [shape = circle, draw] {}; 
\node[vertex] (G-2) at (1.5, 1) [shape = circle, draw] {}; 
\node[vertex] (G-3) at (3.0, 1) [shape = circle, draw] {}; 
\draw[] (G--3) .. controls +(-0.5, 0.5) and +(0.5, 0.5) .. (G--2); 
\draw[] (G--2) .. controls +(-0.5, 0.5) and +(0.5, 0.5) .. (G--1); 
\end{tikzpicture} \ \cdots \ \begin{tikzpicture}[scale = 0.5,thick, baseline={(0,-1ex/2)}] 
\tikzstyle{vertex} = [shape = circle, minimum size = 7pt, inner sep = 1pt] 
\node[vertex] (G--2) at (1.5, -1) [shape = circle, draw] {}; 
\node[vertex] (G--1) at (0.0, -1) [shape = circle, draw] {}; 
\node[vertex] (G-1) at (0.0, 1) [shape = circle, draw] {}; 
\node[vertex] (G-2) at (1.5, 1) [shape = circle, draw] {}; 
\draw[] (G--2) .. controls +(-0.5, 0.5) and +(0.5, 0.5) .. (G--1); 
\end{tikzpicture}}, 
\end{equation}
 with \eqref{PhiHn} extended linearly multiplicatively, and where the partition diagram on the right of \eqref{PhiHn} is of order $n$. The mapping in 
 \eqref{Phidefinition} is an injective, combinatorial Hopf morphism, noting that the equality $\zeta_{\textsf{NSym}} = \zeta_{\textsf{ParSym}} \circ \Phi$ 
 follows in a direct way from \eqref{zetaNSym} and \eqref{zetaParSym}. As below, we introduce a projection morphism from 
 $\textsf{ParSym}$ to $\textsf{NSym}$. 

 Let the statistic $m = m(\pi)$ defined in Theorem \ref{graphtheorem} be extended so that $m(\varnothing) = 0$. For an $\otimes$-irreducible partition 
 diagram $\pi$, let 
\begin{equation}\label{chiproj}
 \chi\colon \textsf{ParSym} \to \textsf{NSym}
\end{equation}
 map $\textsf{H}_{\pi}$ to the complete homogeneous basis element $H_{m(\pi)}$, 
 recalling that 
 $m(\pi)$ is a well defined value according to the uniqueness property 
 given in Theorem \ref{graphtheorem}. 
 We then let $\chi$ be extended linearly and so as to be 
 compatible with \eqref{productParSym}. 

\begin{example}
 Returning to Example \ref{exampleDelta}, if we were to replace each expression of the form $\textsf{H}_{\pi}$ in Example \ref{exampleDelta} with $ 
 \chi(\textsf{H}_{\pi})$, the left-hand side of the equality in Example \ref{exampleDelta} would yield $\Delta H_{2}$, and the right-hand side of the 
 equality in Example \ref{exampleDelta} would yield $H_{2} \otimes H_{0} + H_{1} \otimes H_{1} + H_{0} \otimes H_{2}$. 
 This illustrates \eqref{PhitensorPhi} holding, with respect to $\chi$. 
\end{example}

 Following \cite{BergeronReutenauerRosasZabrocki2008,RosasSagan2006}, we express that 
 much about research that concerns both $\textsf{NSym}$ and $\textsf{NCSym}$ is motivated by the problem 
 of developing a better understanding as to the relationship between $\textsf{NSym}$ and $\textsf{NCSym}$. 
 This motivates the problem of constructing a Hopf algebra that both contains and projects onto $\textsf{NSym}$
 in natural and combinatorially inspired ways, 
 in the hope of elucidating the aforementioned problem, with the use of combinatorial structures and properties related to $\textsf{NCSym}$.
 Since the bases of $\textsf{NCSym}$ are indexed by set partitions, 
 this motivates our construction of a Hopf algebra $\textsf{ParSym}$ that has bases indexed by partition diagrams
 and that contains and projects onto $\textsf{NSym}$. 
 As we recently discussed \cite{CampbellAnn}, it is often more expedient to work with analogues of objects from $\textsf{Sym}$ 
 in the noncommutative setting given by $\textsf{NSym}$; this motivates our lifting $\textsf{ParSym}$ of $\textsf{NSym}$
 and the study of this lifting in relation to $\textsf{NCSym}$. 

\begin{theorem}\label{theoremproj}
 The mapping in \eqref{chiproj} is a combinatorial Hopf morphism. 
\end{theorem}

\begin{proof}
 The mapping $\chi$ is defined to be linear and to preserve the multiplicative operation for $\textsf{ParSym}$, and we find that $\chi(\textsf{H}_{0}) = 
 H_{0}$, so as to provide an algebra morphism. According to Definition \ref{DeltaParSym}, 
 we may obtain that: 
\begin{equation}\label{chichiGG}
 \left( \chi \otimes \chi \right) \circ \Delta \textsf{H}_{\pi} = \sum_{G_{1}, G_{2}} 
 \chi \left( \textsf{H}_{G_{1}} \right) \otimes \chi \left( \textsf{H}_{G_{2}} \right). 
\end{equation}
 According to Theorem \ref{graphtheorem}, we may let $\textsf{H}_{G_{1}}$ and $\textsf{H}_{G_{2}}$, respectively, be uniquely decomposed as $ 
 \bullet$-products of nontrivial and $\bullet$-irreducible
 $\textsf{H}$-generators of lengths $\ell_{1}$ and $\ell_{2}$, so that $\textsf{H}_{\pi}$ is uniquely decomposable as a 
 $\bullet$-product of nontrivial and $\bullet$-irreducible 
 $\textsf{H}$-generators of length $\ell_{1} + \ell_{2}$, which, for fixed $\pi$, we denote as a fixed value $\ell = m(\pi)$. 
 Furthermore, and again by Theorem \ref{graphtheorem}, for the same fixed value $\ell$, and for a given value $\ell_{1}$ in $\{ 0, 1, \ldots, \ell \}$, there is 
 exactly one graph $G_{1}$ such that $\textsf{H}_{\pi} = \textsf{H}_{G_{1}} \bullet \textsf{H}_{G_{2}}$ and such that 
 $\chi\left( \textsf{H}_{{G}_{1}} \right) = H_{\ell_{1}}$, namely, the graph $G_{1}$ such that $\textsf{H}_{G_{1}}$ is equal to the $\bullet$-product of 
 the first $\ell_{1}$ factors in the unique $\bullet$-decomposition of $\textsf{H}_{\pi}$, and this gives us that $\textsf{H}_{G_{2}}$ equals the 
 $\bullet$-product of the remaining factors in the same $\bullet$-decomposition of $\textsf{H}_{\pi}$, and this is according to our characterization of 
 $ \bullet$-decompositions, as in Theorem \ref{theoremcategory}. 
 So, from \eqref{chichiGG}, we may obtain that 
\begin{align*}
 \left( \chi \otimes \chi \right) \left( \Delta \textsf{H}_{\pi} \right) 
 & = \sum_{\ell_{1} + \ell_{2} = \ell} H_{\ell_{1}} \otimes H_{\ell_{2}} \\ 
 & = \Delta\left( \chi\left( \textsf{H}_{\pi} \right) \right). 
\end{align*}
 To obtain that 
 $ \varepsilon_{\textsf{NSym}} \circ \chi = \varepsilon_{\textsf{ParSym}}$, 
 we begin by letting $\pi$ again be reducible, with $\ell$ as before, so that we may write 
\begin{align*}
 \varepsilon_{\textsf{NSym}}\left( \chi\left( \textsf{H}_{\pi} \right) \right) 
 & = \varepsilon_{\textsf{NSym}}\left( H_{\ell} \right) \\ 
 & = \begin{cases} 
 1 & \text{if $\ell = 0$,} \\
 0 & \text{otherwise,} 
 \end{cases} \\
 & = \begin{cases} 
 1 & \text{if $\pi = \varnothing$,} \\
 0 & \text{otherwise.} 
 \end{cases} 
\end{align*}
 By again letting $\pi$ and $\ell$ be as before, we may find that 
\begin{align*}
 \zeta_{\textsf{NSym}} \left( \chi\left( \textsf{H}_{\pi} \right) \right)
 & = \zeta_{\textsf{NSym}} \left( H_{\ell} \right) \\ 
 & = \begin{cases} 
 1 & \text{if $\ell \in \{ 0, 1 \}$}, \\ 
 0 & \text{otherwise}, 
 \end{cases} \\ 
 & = \begin{cases} 
 1 & \text{if $\pi \in \left\{ \varnothing, \begin{tikzpicture}[scale = 0.5,thick, baseline={(0,-1ex/2)}] 
\tikzstyle{vertex} = [shape = circle, minimum size = 7pt, inner sep = 1pt] 
\node[vertex] (G--1) at (0.0, -1) [shape = circle, draw] {}; 
\node[vertex] (G-1) at (0.0, 1) [shape = circle, draw] {}; 
\end{tikzpicture} \right\}$}, \\ 
 0 & \text{otherwise}, 
 \end{cases} 
\end{align*}
 so that $\zeta_{\textsf{PArSym}} = \zeta_{\textsf{NSym}} \circ \chi$, 
 as desired. 
\end{proof}

\subsection{Relationship with $\textsf{QSym}$}
 A fundamental result in the study of CHAs due to Aguiar, Bergeron, and Sottile \cite{AguiarBergeronSottile2006}, described as \emph{beautiful} in 
 \cite{BenedettiHallamMachacek2016} due to the way it relates a quasisymmetric function to each element in a given CHA, may be formulated as in Theorem 
 \ref{fundamentalCHA} below, where $\textsf{QSym}$ denotes the Hopf algebra dual to $\textsf{NSym}$, and where the monomial basis $\{ M_{\alpha} 
 : \alpha \in \mathcal{C} \}$ is dual to \eqref{completefamily}, according to the pairing $\langle \cdot, \cdot \rangle\colon \textsf{NSym} \times 
 \textsf{QSym} \to \mathbbm{k}$ such that $\langle H_{\alpha}, M_{\beta} \rangle = \delta_{\alpha, \beta}$. The character $ \zeta_{\textsf{QSym}}$ is 
 such that $ \zeta_{\textsf{QSym}}(M_{\alpha}) $ agrees with \eqref{zetaNSym}. 

\begin{theorem}\label{fundamentalCHA} \cite[Theorem 4.1]{AguiarBergeronSottile2006} 
 For a CHA 	 $(\mathcal{H}, \zeta)$, there is a unique morphism of CHAs from $ (\mathcal{H}, \zeta)$ 
 to $ (\textsf{\emph{QSym}}, \zeta_{\textsf{\emph{QSym}}})$. 
\end{theorem}

 Since we have proved that $\chi\colon\textsf{ParSym} \to \textsf{NSym}$ 
 is a combinatorial Hopf morphism, 
 and since the proof of Theorem 4.1 from 
 \cite{AguiarBergeronSottile2006} provides a way of constructing a combinatorial Hopf morphism from 
 $\textsf{NSym}$ to $\textsf{QSym}$, 
 by taking the composition of the former CHA morphism evaluated at this latter CHA morphism, 
 we obtain the unique CHA morphism from $\textsf{ParSym}$ to $\textsf{QSym}$. 

\section{Diagram subalgebras and Hopf subalgebras}\label{sectionDiagramHopf}
 A \emph{diagram algebra} may be broadly understood to refer to a subalgebra of $\mathbb{C}A_k(n)$. By taking subalgebras of the underlying algebra of 
 the bialgebra $\textsf{ParSym}$ by restricting the diagrams allowed for index sets of the form indicated in \eqref{gradedcomponent}, e.g., according to 
 a given family of diagram algebras, this gives rise to CHA structures worthy of study in relation to $\textsf{ParSym}$. The work of Colmenarejo 
 et al.\ in \cite{ColmenarejoOrellanaSaliolaSchillingZabrocki2020} was devoted to the lifting of combinatorial properties associated with the symmetric 
 group algebra $\mathbb{C}S_k$ as a subalgebra of $\mathbb{C}A_k(n)$ to diagram subalgebras of $\mathbb{C}A_k(n)$ apart from $\mathbb{C}S_k$, 
 and this leads us to consider how the partition diagrams indexing the bases of these subalgebras could lead to \emph{Hopf} subalgebras of $ 
 \textsf{ParSym}$, as opposed to \emph{diagram}
 subalgebras of $\mathbb{C}A_k(n)$. For each of the diagram subalgebras considered in 
 \cite{ColmenarejoOrellanaSaliolaSchillingZabrocki2020}, with reference to Table 1 in \cite{ColmenarejoOrellanaSaliolaSchillingZabrocki2020}, we 
 consider, as below, a corresponding subalgebra of $\textsf{ParSym}$. A remarkable property about our new Hopf algebra $\textsf{ParSym}$ is given 
 by how each of the families of partition diagrams associated with what may be considered as the main or most important diagram subalgebras 
 \cite{ColmenarejoOrellanaSaliolaSchillingZabrocki2020} naturally gives rise to a Hopf subalgebra of $\textsf{ParSym}$. This nicely illustrates how the 
 morphisms $\Delta$ and $S$ we have defined on $\textsf{ParSym}$ are natural, in terms of lifting combinatorial objects and properties associated with 
 both $\mathbb{C}A_k(n)$ and $\textsf{NSym}$. 

 The \emph{propagation number} of a partition diagram $\pi$ refers to the number of components in $\pi$ that contain at least one upper vertex and at least 
 one lower vertex. So, there is a natural correspondence between the partition diagrams in $A_{i}$ of propagation number $i$ and the permutations in 
 the symmetric group $S_{i}$. So, if we take $S_{i}$ as a submonoid of $A_{i}$, and then define $\mathbbm{k}$-spaces obtained by replacing the index 
 set $A_{i}$ in \eqref{gradedcomponent} with $S_{i}$, and then form a graded subalgebra of \eqref{ParSymdefinition}, subject to the same 
 multiplication rule in \eqref{productParSym}, this leads us to consider the effect of our coproduct operation for $\textsf{ParSym}$, restricted to the 
 graded algebra on permutations we have formed. For a permutation $p$ written as a permuting diagram $\{ \{ 1, p(1)' \}, \{ 2, p(2)' \}, \ldots, \{ k, p(k)' 
 \} \}$, by writing this diagram as a concatenation of $\otimes$-irreducible diagrams, the resultant factors are given by irreducible, permuting diagrams, 
 and we find that 
 these factors are primitive, since no two bottom nodes of a permuting diagram can be in the same component, so we find that we obtain 
 closure under $\Delta$, and an equivalent argument applies to obtain closure with repsect to the antipode $S$. The subalgebra of $ 
 \textsf{ParSym}$ spanned by permutations is not isomorphic to the Malvenuto--Reutenauer Hopf algebra of permutations, which involves the shuffle 
 product of permutations, as opposed to the concatenation of permutation diagrams, for the definition of its product operation. The number of 
 generators in each degree for the subalgebra spanned by permutations we have defined is given by the Boolean transform of the sequence of factorials, 
 indexed in the OEIS as A003319, and this sequence has often arisen in the context of Hopf algebras and Hopf monoids. This motivates the exploration 
 of diagram superalgebras of $\mathbb{C}S_k$ in $\mathbb{C}A_k(n)$, in relation to $\textsf{ParSym}$. 

 \emph{Planar} diagrams refer to partition diagrams that may be expressed as planar graphs, and the dimension of the planar subalgebra of 
 $\mathbb{C}A_k(n)$ is $\frac{1}{2k+1} \binom{4k}{2k}$. It seems that the Boolean transform for the corresponding integer sequence, giving 
 $\frac{1}{4n-1} \binom{4n}{2n}$, has not been considered in the context of Hopf algebras. This motivates the following result. For the sake of 
 convenience, we may write, as below, the expression $\pi$ in place of $\textsf{H}_{\pi}$. 

\begin{theorem}\label{theoremplanar}
 The graded subalgebra of $\textsf{\emph{ParSym}}$ spanned by planar diagrams forms a Hopf subalgebra. 
\end{theorem}

\begin{proof}
 Let $\pi$ and $\rho$ be planar diagrams. Since $\pi$ and $\rho$ do not have edge crossings, placing $\rho$ to the right of $\pi$ would not result in an edge 
 crossing, so $\pi \otimes \rho$ is planar. Letting $\pi$ be as before, suppose that $\pi$ may be expressed so that $\pi = \rho^{(1)} \bullet \rho^{(2)}$ 
 for nonempty diagrams $\rho^{(1)}$ and $\rho^{(2)}$. We may add, if necessary, an edge in $\pi$ incident with 
 the lower right 
 vertex of $\rho^{(1)}$ and the lower left vertex of $\rho^{(2)}$. Since this added edge is on the border of $\pi$, this would not result in any crossing 
 edges. Mimicking our proof of Theorem \ref{graphtheorem}, we may remove edges from $\pi$ in such a way so that $\pi$ may be formed by adding an 
 isthmus incident with the lower right node of $\rho^{(1)}$ and the lower left node of $\rho^{(2)}$, and the removal of any edges would not result in 
 any crossing edges. Since $\pi$ is equivalent to the diagram given by taking $\rho^{(1)}$, placing $\rho^{(2)}$ to the right of $\rho^{(1)}$, and adding 
 an edge at the border, since this process would not create any edge crossings, we may conclude that there are no edge crossings in $\rho^{(1)}$ and no 
 edge crossings in $\rho^{(2)}$, by planarity of $\pi$. So, by expanding $\Delta \pi$ according to Definition \ref{DeltaParSym}, we would find that 
 each resultant term $ \textsf{H}_{G_{1}} \otimes \textsf{H}_{G_{2}}$ would be such that $G_{1}$ and $G_{2}$ are necessarily planar as diagrams. 
 So, we obtain closure with respect to $\Delta$, and an equivalent argument gives us closure with respect to $S$. The specified closure properties yield the 
 desired result. 
\end{proof}

 Following the ordering as to how diagram subalgebras are introduced in \cite{ColmenarejoOrellanaSaliolaSchillingZabrocki2020}, we proceed to define a 
 \emph{matching} as a partition diagram $\pi$ such that all blocks in $\pi$ are of size at most two. The dimension of the corresponding subalgebra of $ 
 \mathbb{C}A_k(n)$, i.e., the Rook-Brauer algebra of order $k$, is $\sum_{i=0}^{k} \binom{2k}{2i} (2i-1)!!$. It seems that the Boolean transform 
 of the corresponding integer sequence has not previously been considered, which motivates the following result. 

\begin{theorem}\label{theoremperfect}
 The graded subalgebra of $\textsf{\emph{ParSym}}$ spanned by matchings forms a Hopf subalgebra. 
\end{theorem}

\begin{proof}
 For matchings $\pi$ and $\rho$, by placing $\rho$ to the right of $\pi$, this does not alter the cardinalities of the blocks in $\pi$ or the blocks in $\rho$, 
 so the concatenation $\pi \otimes \rho$ is such that each block is of size one or two. Letting $\pi$ be as before, suppose that $\pi = \rho^{(1)} \bullet 
 \rho^{(2)}$. By definition of the operation denoted as $\bullet$, the lower right vertex of $\rho^{(1)}$ is in the same component as the lower left 
 vertex of $\rho^{(2)}$, in the diagram $\pi$. Since $\pi$ is a matching, we may deduce that the aforementioned component consists of two vertices. 
 So, with regard to our proof of Theorem \ref{graphtheorem}, we would not have to remove any edges to form an isthmus. So, the diagram $\pi$ may be 
 obtained by taking $\rho^{(1)}$, placing $\rho^{(2)}$ beside $\rho^{(1)}$, and then forming an isthmus at the border that forms a connected component of 
 size two. So, this process would not result in any connected components of size greater than two, so $\rho^{(1)}$ and $\rho^{(2)}$ are matchings. 
 Mimicking a line of reasoning from our proof of Theorem \ref{theoremplanar}, we obtain closure with respect to $\Delta$ and $S$. 
\end{proof}

 A \emph{perfect matching} is a matching such that each block is of size two. 
 The diagram subalgebra of $\mathbb{C}A_k(n)$ spanned/indexed by perfect matchings 
 is the famous Brauer algebra, which is of dimension $(2k-1)!!$. 
 The Boolean transform, in this case, agrees with the OEIS sequence A000698, 
 which is associated with many enumerative interpretations. 

\begin{theorem}
 The graded subalgebra of $\textsf{\emph{ParSym}}$ spanned by perfect matchings forms a Hopf subalgebra. 
\end{theorem}

\begin{proof}
 We may mimic $\otimes$-closure argument in our proof of Theorem \ref{theoremperfect} to demonstrate the desired $\otimes$-closure property for perfect 
 matchings. Now, for a \emph{perfect} matching $\pi$, we assume, by way of contradiction, that there exist nonempty diagrams such that $\pi = 
 \rho^{(1)} \bullet \rho^{(2)}$. Again, by definition of the operation $\bullet$, the lower right vertex of $\rho^{(1)}$ is in the same component as the 
 lower left vertex of $\rho^{(2)}$, in $\pi$. So, since each connected component in $\pi$ is of size two, we may deduce that there is a component in $\pi$ 
 consisting entirely of the lower right vertex of $\rho^{(1)}$ and the lower left vertex of $\rho^{(2)}$. Furthermore, we may deduce that the upper right 
 vertex of $\rho^{(1)}$ cannot be in the same component as the upper left vertex of $\rho^{(2)}$, in $\pi$, because, otherwise, since each component 
 of $\pi$ is of size two, these two upper vertices would form an edge, say $\{ i, i+1 \}$, but then both $\{ i, i+1 \}$ and $\{ i', (i+1)' \}$ would be 
 components in $\pi$, which is impossible, since $\bullet$ would only have the effect of adding a bottom isthmus forming a component of size two, and this 
 could not have any effect in terms of adding an upper edge. So, by assumption that $\rho^{(1)}$ and $\rho^{(2)}$ are nonempty, each of $\rho^{(1)}$ 
 and $\rho^{(2)}$ has at least two nodes, but since a given partition diagram has an even number of nodes, there would be an odd number of nodes 
 ``available'' in $\rho^{(1)}$ apart from the lower right node that would be adjacent, in $\pi$, with the lower left node of $\rho^{(2)}$, but then we 
 could not form a perfect matching from these odd number of ``remaining'' nodes. So, for a perfect matching $\pi$, we may deduce that $\pi$ is primitive in 
 $\textsf{ParSym}$, so that we obtain closure with respect to $\Delta$. An equivalent argument applies to closure with respect to $S$. 
\end{proof}

 A \emph{patial permutation} is a partition diagram $\pi$ such that each block of $\pi$
 is of size one or two and such that each block that is of size two in $\pi$ is propagating, i.e., 
 so that each such block has at least one upper vertex and at least one lower vertex. 
 The diagram subalgebra of $\mathbb{C}A_k(n)$ spanned or indexed by partial permutations is of 
 dimension $\sum_{i=0}^{k} \binom{k}{i}^2 i!$, and it seems that the Boolean transform for 
 such expressions has not previously been considered. 

\begin{theorem}
 The graded subalgebra of $\textsf{\emph{ParSym}}$
 spanned by perfect matchings forms a Hopf subalgebra. 
\end{theorem}

\begin{proof}
 This may be proved in effectively the same way as in with the case for permuting diagrams. 
\end{proof}

 Mimicking our above proof on perfect matchings, we may obtain 
 corresponding closure properties with respect to planar perfect matchings, 
 which are associated with Temperley--Lieb algebras, 
 and similarly with respect to the other major diagram algebras, as in with the Motzkin algebra and the planar rook algebra. 

\section{Conclusion}\label{sectionConclusion}
 We conclude by briefly describing some future areas of research related to $\textsf{ParSym}$. 

 If we were to construct a CHA generated by $\otimes$-irreducible partition diagrams, but with a \emph{commutative} operation given by the disjoint 
 union of graphs in place of the noncommutative $\otimes$ operation, then this would give rise to a free commutative algebra and an analogue of 
 partition diagrams, given by the equivalence classes given by identifying two partition diagrams if such diagrams may be obtained by permuting the 
 positions of $\otimes$-irreducible factors. We encourage the study of this free commutative algebra and its relation to Schur--Weyl duality. Given how 
 partition diagrams are derived using the centralizer algebra in \eqref{maincentralizer}, 
 what would be the appropriate analogy of this Schur--Weyl duality corresponding to the ``commutative'' version of partition diagrams 
 we have suggested? 

 By analogy with how the stack-sorting map was lifted from permutations to partition diagrams in \cite{Campbellunpublished}, how could the shuffle 
 product of permutations be lifted to partition diagrams, to form a new CHA on partition diagrams, but with a shuffle-type product being used in place 
 of concatenation? This would be of interest in terms of the problem of lifting the Malvenuto--Reutenauer Hopf algebra of permutations, with the 
 use of partition diagrams. 

 Instead of constructing a CHA via the horizontal concatenation of partition diagrams, as above, how 
 could a subalgebra of $\mathbb{C}A_{k}(n)$, or a related 
 structure such as the rook monoid algebra, that does not have a group algebra structure be endowed with a Hopf algebra structure, i.e., with a 
 vertically defined 
 diagram multiplication operation, and by analogy with the Hopf algebra structure 
 on the group algebra $\mathbb{C}S_k$? 

 What is the Hopf algebra $\textsf{ParSym}^{\ast}$ that is dual to $\textsf{ParSym}$, i.e., how can its product and coproduct operations be defined in an 
 explicit, combinatorial way, by analogy with how the CHA $\textsf{QSym}$ of quasisymmetric functions is dual to $\textsf{NSym}$? 

 Partition algebras are considered to have two distinguished bases: The diagram basis and the orbit basis. If we think of a given $\textsf{H}$-basis element 
 $\textsf{H}_{\pi}$ as being in correspondence with the diagram basis element $d_{\pi}$, then one may thus construct an analogue of the orbit basis for 
 $\textsf{ParSym}$. How could this orbit-like basis be applied in relation to the CHA structure on $\textsf{ParSym}$? 

 We have introduced liftings of the $H$- and $E$- bases of $\textsf{NSym}$, as in with the $\textsf{H}$- and $\textsf{E}$-bases of $\textsf{ParSym}$. 
 How could the other distinguished bases of $\textsf{NSym}$ be lifted to $\textsf{ParSym}$ in a way that is applicable to the study of the structure of 
 partition algebras? 

 The statistic $m = m(\pi)$ involved in Theorem \ref{graphtheorem} may be of interest in its own right, in view of its uses, as above, in the context of the 
 study of the structure of $\textsf{ParSym}$. How can $m(\pi)$ be computed in an efficient way, in view of the algorithmic nature of our proof of 
 Theorem \ref{graphtheorem}, and what interesting properties can be derived from 
 or associated with the number of diagrams $\pi$ of a fixed order $n$ 
 such that $m(\pi) = k$ for fixed $k$? 

 Recall that we have offered a reformulation as to how $\textsf{NSym}$ may be defined or constructed, as in Section \ref{subsectionNear}. This is of 
 interest due to how, in contrast to how $\textsf{NSym}$ may be defined as a free algebra with one generator in \emph{each} positive integer degree, our 
 reformulation of $\textsf{NSym}$ is given by how it is generated by a \emph{single} element, subject to the two operations $\circ$ and $\bullet$ indicated 
 in Section \ref{subsectionNear}. It seems that $\textsf{NSym}$ has not been previously been considered in this way, i.e., as being freely generated by 
 a \emph{single} object. How could this be explored in a category-theoretic way and with regard to how $\textsf{NSym}$ is universal in the category of 
 CHAs, and in regard to the connection between $\textsf{NSym}$ and Grothendieck rings for finitely generated 
 projective representations \cite{BergBergeronSaliolaSerranoZabrocki2014}? 

 It appears that: Endowed with the products $\otimes$ and $\bullet$, we have that $\textsf{ParSym}$
 has the structure of a matching associative algebra, according to the definition of this term given in 
 \cite{ZhangGaoGuo2020}, so that, for partition diagrams $x$, $y$, and $z$, we have that 
\begin{align*}
 (x \bullet y) \bullet z & = x \bullet (y \bullet z), \\ 
 (x \bullet y) \otimes z & = x \bullet (y \otimes z), \\ 
 (x \otimes y) \bullet z & = x \otimes (y \bullet z), \\ 
 (x \otimes y) \otimes z & = x \otimes (y \otimes z). 
\end{align*}
 It seems that $\textsf{ParSym}$ is freely generated by partition diagrams that are both $\otimes$- and $\bullet$-irreducible, 
 and we encourage the exploration of this. 

\subsection*{Acknowledgements}
 The author is grateful to acknowledge support from a Killam Postdoctoral Fellowship. 

\subsection*{Competing interests statement}
 The author has no competing interests to declare.

 \

 Dalhousie University 

 Department of Mathematics and Statistics

{\tt jmaxwellcampbell@gmail.com}


\begin{thebibliography}{99}

\bibitem{offthecharts}
 M. Aguiar, C. Andr{\'e}, C. Benedetti, N. Bergeron, Z. Chen, P. Diaconis, A. Hendrickson, S. Hsiao, I. M. Isaacs, A. Jedwab, 
 K. Johnson, G. Karaali, A. Lauve, T. Le, S. Lewis, H. Li, K. Magaard, E. Marberg, J.-C. Novelli, A. Pang, F. Saliola, L. Tevlin, 
 J.-Y. Thibon, N. Thiem, V. Venkateswaran, C. R. Vinroot, N. Yan, M. Zabrocki, 
 Supercharacters, symmetric functions in noncommuting variables, and related {Hopf} algebras, 
 Adv. Math.\ 229 (2012) 2310--2337. 

\bibitem{AguiarBergeronSottile2006}
 M. Aguiar, N.\ Bergeron, F. Sottile, 
 Combinatorial {Hopf} algebras and generalized {Dehn}-{Sommerville} relations, 
 Compos. Math.\ 142 (2006) 1--30. 

\bibitem{AguiarMahajan2014}
 M.\ Aguiar, S.\ Mahajan, 
 On the Hadamard product of Hopf monoids, 
 Can. J. Math.\ 66 (2014) 481--504. 

\bibitem{AguiarOrellana2008}
 M.\ Aguiar, R.\ C.\ Orellana, 
 The Hopf algebra of uniform block permutations, 
 J. Algebr. Comb.\ 28 (2008) 115--138. 

\bibitem{AliniaeifardLivanWilligenburg2022}
 F.\ Aliniaeifard, S.\ X.\ Li, S. van Willigenburg, 
 Schur functions in noncommuting variables, 
 Adv. Math.\ 406 (2022) Id/No 108536. 

\bibitem{ArcisMarquez2022}
 D.\ Arcis, S.\ M{\'a}rquez, 
 Hopf algebras on planar trees and permutations, 
 J. Algebra Appl.\ 21 (2022) Id/No 2250224. 

\bibitem{BenedettiHallamMachacek2016}
 C.\ Benedetti, J. Hallam, J.\ Machacek, 
 Combinatorial {Hopf} algebras of simplicial complexes, 
 SIAM J. Discrete Math.\ 30 (2016) 1737--1757. 

\bibitem{BenedettiSagan2017}
 C.\ Benedetti, B.\ E.\ Sagan, 
 Antipodes and involutions, 
 J. Combin. Theory Ser. A 148 (2017) 275--315. 

\bibitem{BenkartWitherspoon2004}
 G.\ Benkart, S.\ Witherspoon, 
 Representations of two-parameter quantum groups and {Schur}-{Weyl} duality, in: 
 Hopf algebras. Proceedings from the international conference, DePaul University, Chicago, IL, USA held during the 2001--2002 academic year; 
 New York, NY: Marcel Dekker, 2004, pp.\ 65--92. 

\bibitem{BergBergeronSaliolaSerranoZabrocki2014}
 C.\ Berg, N.\ Bergeron, F.\ Saliola, L.\ Serrano, M.\ Zabrocki, 
 A lift of the {Schur} and {Hall}-{Littlewood} bases to non-commutative symmetric functions, 
 Can. J. Math.\ 66 (2014) 525--565. 

\bibitem{BergeronGonzalezdLeonLiPangVargas2023}
 N.\ Bergeron, R.\ S.\ Gonz{\'a}lez d'Le{\'o}n, S.X.\ Li, C.\ Y.\ A.\ Pang, Y.\ Vargas, 
 Hopf algebras of parking functions and decorated planar trees, 
 Adv. Appl. Math.\ 143 (2023) Id/No 102436. 

\bibitem{BergeronHohlwegRosasZabrocki2006}
 N.\ Bergeron, C.\ Hohlweg, M.\ Rosas, M.\ Zabrocki, 
 Grothendieck bialgebras, partition lattices, and symmetric functions in noncommutative variables, 
 Electron. J. Comb.\ 13 (2006) r75, 19. 

\bibitem{BergeronReutenauerRosasZabrocki2008}
 N.\ Bergeron, C.\ Reutenauer, M. Rosas, M.\ Zabrocki, 
 Invariants and coinvariants of the symmetric group in noncommuting variables, 
 Can. J. Math.\ 60 (2008) 266--296. 

\bibitem{BergeronThiem2013}
 N.\ Bergeron, N. Thiem, 
 A supercharacter table decomposition via power-sum symmetric functions, 
 Int. J. Algebra Comput. 23 (2013) 763--778. 

\bibitem{BessenrodtLuotovanWilligenburg2011}
 C.\ Bessenrodt, K. Luoto, S. van Willigenburg, 
 Skew quasisymmetric {Schur} functions and noncommutative {Schur} functions, 
 Adv. Math.\ 226 (2011) 4492--4532. 

\bibitem{Campbellunpublished}
 J.\ M.\ Campbell, 
 A lift of {W}est's stack-sorting map to partition diagrams, 
 Pacific J. Math.\ 324 (2023) 227--248. 

\bibitem{CampbellAnn}
 J.\ M.\ Campbell, 
 On antipodes of immaculate functions, 
 	Ann.\ Comb.\ (2023). 

\bibitem{CheballahGiraudoMaurice2015}
 H.\ Cheballah, S.\ Giraudo, R.\ Maurice, 
 Hopf algebra structure on packed square matrices, 
 J. Comb. Theory, Ser. A 133 (2015) 139--182. 

\bibitem{ColmenarejoOrellanaSaliolaSchillingZabrocki2020}
 L.\ Colmenarejo, R.\ Orellana, F.\ Saliola, A.\ Schilling, M.\ Zabrocki, 
 An insertion algorithm on multiset partitions with applications to diagram algebras, 
 J. Algebra 557 (2020) 97--128. 

\bibitem{Comtet1972}
 L.\ Comtet, 
 Sur les coefficients de l'inverse de la s{\'e}rie formelle {{\(\sum n! t^n\)}}, 
 C. R. Acad. Sci., Paris, S{\'e}r. A 275 (1972) 569--572. 

\bibitem{DaughertyHerbrich2014}
 Z.\ Daugherty, P.\ Herbrich, 
 Centralizers of the infinite symmetric group, in: 
 Proceedings of the 26th international conference 
 on formal power series and algebraic combinatorics, FPSAC 2014, Chicago, IL, USA, June 29 -- July 3, 2014; 
 Nancy: The Association. Discrete Mathematics \& Theoretical Computer Science (DMTCS), 
 2014, pp.\ 453--464. 
 
\bibitem{DuchampLuqueNovelliTolluToumazet2011}
 G.\ H.\ E.\ Duchamp, J.-G. Luque, J.-C. Novelli, C. Tollu, F. Toumazet, 
 Hopf algebras of diagrams, 
 Int. J. Algebra Comput.\ 21 (2011) 889--911. 

\bibitem{Dupont2014}
 C.\ Dupont, 
 The combinatorial {Hopf} algebra of motivic dissection polylogarithms, 
 Adv. Math.\ 264 (2014) 646--699. 

\bibitem{Fisher2010}
 F.\ Fisher, 
 CoZinbiel hopf algebras in combinatorics, 
 PhD Thesis, George Washington University, 2010. 

\bibitem{Foissy2007}
 L.\ Foissy, 
 Bidendriform bialgebras, trees, and free quasi-symmetric functions, 
 J. Pure Appl. Algebra 209 (2007) 439--459. 

\bibitem{Foissy2012}
 L.\ Foissy, 
 Free and cofree {Hopf} algebras, 
 J. Pure Appl. Algebra 216 (2012) 480--494. 

\bibitem{Foissy2002}
 L.\ Foissy, 
 The Hopf algebras of decorated rooted trees. {I}, 
 Bull. Sci. Math.\ 126 (2002) 193--239. 

\bibitem{FoissyUnterberger2013}
 L.\ Foissy, J.\ Unterberger, 
 Ordered forests, permutations, and iterated integrals, 
 Int. Math. Res. Not.\ 2013 (2013) 846--885. 

\bibitem{GaoKitaevZhang2018}
 A.\ L.\ L.\ Gao, S. Kitaev, P.\ B.\ Zhang, 
 On pattern avoiding indecomposable permutations, 
 Integers 18 (2018) paper a2, 23. 

\bibitem{GelfandKrobLascouxLeclercRetakhThibon1995}
 I.\ M.\ Gelfand, D.\ Krob, A.\ Lascoux, B.\ Leclerc, 
 V.\ S.\ Retakh, J.-Y.\ Thibon, 
 Noncommutative symmetric functions, 
 Adv. Math.\ 112 (1995) 218--348. 

\bibitem{Grinberg2017}
 D.\ Grinberg, 
 Dual creation operators and a dendriform algebra structure on the quasisymmetric functions, 
 Can. J. Math.\ 69 (2017) 21--53. 

\bibitem{GrossmanLarson1989}
 R.\ Grossman, R.\ G.\ Larson, 
 Hopf-algebraic structure of families of trees, 
 J.\ Algebra 126 (1989) 184--210. 

\bibitem{Halverson2001}
 T.\ Halverson, 
 Characters of the partition algebras, 
 J. Algebra 238 (2001) 502--533. 

\bibitem{HalversonRam2005}
 T.\ Halverson, A.\ Ram, 
 Partition algebras, 
 Eur. J. Comb.\ 26 (2005) 869--921. 

\bibitem{HumpertMartin2012}
 B.\ Humpert, J.\ L.\ Martin, 
 The incidence Hopf algebra of graphs, 
 Eur. J. Comb.\ 26 (2012) 869--921. 

\bibitem{King2006}
 A.\ King, 
 Generating indecomposable permutations, 
 Discrete Math.\ 306 (2006) 508--518. 

\bibitem{Kreimer2010}
 D.\ Kreimer, 
 The core {Hopf} algebra, in: 
 Quanta of maths. Conference on non commutative geometry in honor of Alain Connes, Paris, France, March 29--April 6, 2007,
 in: American Mathematical Society (AMS); Cambridge, MA: Clay Mathematics Institute 
 (Providence, RI), 2010, pp.\ 313--321

\bibitem{Kreimer1998}
 D.\ Kreimer, 
 On the Hopf algebra strucutre of perturbative quantum field theories, 
 Adv. Theor. Math. Phys.\ 2 (1998) 303--334. 

\bibitem{Li2018}
 S.\ X.\ Li, 
 Theta maps for combinatorial Hopf algebras, 
 PhD Thesis, York University, 2018. 
 
\bibitem{LodayRonco1998}
 J.-L.\ Loday, M.\ O.\ Ronco, 
 Hopf algebra of the planar binary trees, 
 Adv. Math.\ 139 (1998) 293--309. 

\bibitem{Maltcev2007}
 V.\ Maltcev, 
 On a new approach to the dual symmetric inverse monoid {{\(\mathcal I_X^*\)}}, 
 Int. J. Algebra Comput.\ 17 (2007) 567--591. 
 
\bibitem{Mammez2020}
 C.\ Mammez, 
 On the combinatorics of the {Hopf} algebra of dissection diagrams, 
 J. Algebr. Comb.\ 51 (2020) 479--526. 

\bibitem{Maurice2013}
 R.\ Maurice, 
 A polynomial realization of the {Hopf} algebra of uniform block permutations, 
 Adv. Appl. Math.\ 51 (2013) 285--308. 

\bibitem{NovelliPatrasThibon2013}
 J.-C.\ Novelli, F.\ Patras, J.-Y.\ Thibon, 
 Natural endomorphisms of quasi-shuffle {Hopf} algebras, 
 Bull. Soc. Math. Fr.\ 141 (2013) 107--130. 
 
\bibitem{Rey2007}
 M.\ Rey, 
 A self-dual Hopf algebra on set partitions, 2007.

\bibitem{RosasSagan2006}
 M. H.\ Rosas, B. E. Sagan, 
 Symmetric functions in noncommuting variables, 
 Trans. Am. Math. Soc.\ 358 (2006) 215--232. 

\bibitem{Schmitt1993}
 W. R.\ Schmitt, 
 Hopf algebras of combinatorial structures, 
 Can. J. Math.\ 45 (1993) 412--428. 

\bibitem{Schmitt1994}
 W. R.\ Schmitt, 
 Incidence Hopf algebras, 
 J. Pure Appl. Algebra 96 (1994) 299--330. 

\bibitem{Takeuchi1971}
 M.\ Takeuchi, 
 Free {H}opf algebras generated by coalgebras, 
 J. Math. Soc. Japan 23 (1971) 561--582. 

\bibitem{Thiem2010}
 N.\ Thiem, 
 Branching rules in the ring of superclass functions of unipotent upper-triangular matrices, 
 J. Algebr. Comb.\ 31 (2010) 267--298. 
 
\bibitem{WangXuGao2020}
 X.\ Wang, S.-J.\ Xu, X. Gao, 
 A {Hopf} algebra on subgraphs of a graph, 
 J. Algebra Appl.\ 19 (2020) Id/No 2050164. 

\bibitem{ZhangXuGuo2022}
 X.\ Zhang, A. Xu, L. Guo, 
 Hopf algebra structure on free {Rota}-{Baxter} algebras by angularly decorated rooted trees, 
 J. Algebr. Comb.\ 55 (2022) 1331--1349. 

\bibitem{ZhangGaoGuo2020}
 Y.\ Zhang, X. Gao, L. Guo, 
 Matching {R}ota-{B}axter algebras, matching dendriform algebras and matching pre-{L}ie algebras, 
 J. Algebra 552 (2020) 134--170. 

\end{thebibliography}
\end{document}